\documentclass[12pt]{amsart}
\usepackage[utf8]{inputenc}
\usepackage{extarrows}
\usepackage{enumerate}
\usepackage[toc]{appendix}

\usepackage{amsfonts,stmaryrd}
\usepackage{amsgen}
\usepackage{amsmath}
\usepackage{amsthm}
\usepackage{amssymb,amscd}

\usepackage{mathrsfs}
\usepackage{mathtools}
\usepackage{setspace}

\usepackage[all]{xy}
\usepackage{tikz-cd}
\numberwithin{equation}{section}

\newtheorem{theorem}{Theorem}[section]
\newtheorem{lemma}[theorem]{Lemma}
\newtheorem{corollary}[theorem]{Corollary}
\newtheorem{proposition}[theorem]{Proposition}

\theoremstyle{definition}
\newtheorem{definition}[theorem]{Definition}

\theoremstyle{remark}
\newtheorem{remark}[theorem]{Remark}

\allowdisplaybreaks
\DeclareMathOperator\fitzero{Fitt}

\DeclareMathOperator\kernel{Ker}
\DeclareMathOperator\im{Im}

	\title[Fitting Ideals of Projective Limits over Iwasawa Algebras]{Fitting Ideals of Projective Limits of Modules over Non-Noetherian Iwasawa Algebras}
\author{Cristian D. Popescu \and Wei Yin}
\date{}
\AtEndDocument{\bigskip{\footnotesize
		\noindent
		\textsc{Cristian D. Popescu. University of California, San Diego,  Department of Mathematics 0112, 9500 Gilman Drive, La Jolla, CA, 92093,  USA.}
		
		\textit{E-mail address}: \texttt{cpopescu@ucsd.edu}
		\par
		\addvspace{\medskipamount}
		\noindent
		\textsc{Wei Yin. Shanghai Institute for Mathematics and Interdisciplinary             Sciences (SIMIS), Shanghai 200433, China. \quad   Research Institute of Intelligent Complex Systems, Fudan University, Shanghai 200433, China.}
		
		\textit{E-mail address}: \texttt{weiyin101@simis.cn}
}}

\begin{document}

	\maketitle

	\begin{abstract}
		In \cite{grku1}, Greither and Kurihara proved a theorem about the commutativity of projective limits and Fitting ideals for modules over the classical equivariant Iwasawa algebra $\Lambda_G\coloneqq\mathbb{Z}_p[G][[T]]$, where $G$ is a finite, abelian group and $\Bbb Z_p$ is the ring of $p$--adic integers, for some prime $p$. In this paper, we generalize their result first to the Noetherian Iwasawa algebras $\mathcal O[[T_1, T_2, \dots, T_n]]$ and, most importantly, to non-Noetherian algebras $\mathcal O[[T_1, T_2, \dots, T_n, \dots]]$ of countably many generators, with more general rings of coefficients $\mathcal O$. The latter generalization is motivated by the recent work of Bley--Popescu \cite{bleypop} on the Geometric Equivariant Iwasawa Conjecture for function fields, as well as by the emerging Iwasawa theory of Taelman class--modules associated to Drinfeld modules, where the Iwasawa algebras are not Noetherian, of the type described above. A sample application of our results to non--Noetherian geometric Iwasawa theory is given in Appendix B. Further number theoretic applications will be given in an upcoming paper.
	\end{abstract}

	\tableofcontents
	\section{Preliminaries}
In this section, we cover a few basic facts on projective limits and Fitting ideals, which are the main objects of study of this paper. In particular, we will state the theorem of Greither and Kurihara (see \cite{grku1}) which we aim to generalize.
	
	\subsection{Exactness of Projective Limits}
	\begin{definition}
		Let $R$ be a commutative ring, $\{M_i,\varphi_{ji}\}$ be a projective system of $R$-modules indexed by integers, with transition maps $\varphi_{ji}: M_j\to M_i$ for $j\geqslant i$. Recall that we say the projective system satisfies the Mittag-Leffler condition if for any $i$, the family $\{\varphi_{ji}(M_j)\}_{j\geqslant i}$ of submodules of $M_i$ is eventually stationary.
	\end{definition}
	\begin{remark}
		 In particular, the Mittag-Leffler hypothesis is satisfied if either one of the following conditions is satisfied.
		
		\begin{itemize}
			\item[(1)] All the morphisms $\varphi_{ji}$'s are surjective.
		
			\item[(2)] All the modules $M_i$ are of finite length.
		\end{itemize}
	\end{remark}

	The Mittag-Leffler condition is important because of the following well-known result.
	
	\begin{proposition}[see \cite{lang}, Prop. 10.3]
		Let $\{A_i\}_i, \{B_i\}_i, \{C_i\}_i$ be three projective systems of abelian groups indexed by $i\in\mathbb{N}$, and let
		$$0\to A_i \xrightarrow{f_{i} }B_i \xrightarrow{g_{i} }C_i\to0$$ be an exact sequence of the projective systems, in the obvious sense. If $(A_i)_i$ satisfies the Mittag-Leffler condition, then the sequence
		$$0\to \varprojlim A_i \xrightarrow{f}\varprojlim B_i \xrightarrow{g }\varprojlim C_i\to0$$
		is exact, where the maps $f$ and $g$ are the projective limits of $\{f_i\}$ and $\{g_i\}$, respectively.
	\end{proposition}
	We shall make frequent use of the following topological version of this result.
	\begin{lemma}\label{compactinverse}
		Let $\{A_i\}_i, \{B_i\}_i, \{C_i\}_i$ be three projective systems of compact, Hausdorff topological groups indexed by $i\in\mathbb{N}$. Assume that all the transition maps are continuous. Suppose that we have an exact sequence of  projective systems in the category of topological groups:
		$$0\to A_i \xrightarrow{f_{i} }B_i \xrightarrow{g_{i} }C_i\to0.$$ Then we have an exact sequence of topological groups
		$$0\to \varprojlim A_i \xrightarrow{f}\varprojlim B_i \xrightarrow{g }\varprojlim C_i\to0.$$
		
	\end{lemma}
	\begin{proof}
		See \cite{washington}, Lemma 15.16. We remark here that the  proof in loc. cit. implicitly used the Hausdorff condition at the very end, so we imposed that condition in our statement of the Lemma, although many authors work with the convention that a compact topological group is Hausdorff by definition.
	\end{proof}
	\begin{remark}
		The above results remain true if we replace the index set of natural numbers $\mathbb{N}$ by a directed, countable set $I$. However, for simplicity, in this paper we assume that the index sets are the natural numbers.
	\end{remark}

\subsection{Commutativity of Projective Limits and Tensor Products}

	\begin{lemma}[see \cite{emerton}]\label{EmertonTrick}
		Let $R$ be a commutative ring, $N$ be a finitely presented $R$-module, and $\{M_i\}_i$ be a projective system of $R$-modules such that for each index $i$, $M_i$ is of finite length as an $R$-module.		Then, the canonical $R$-module homomorphism $$N\otimes_R (\varprojlim\limits_{i}M_i)\rightarrow\varprojlim\limits_{i}(N\otimes_R M_i ) $$
		is an isomorphism.
	\end{lemma}
	\begin{proof}
		For the convenience of the reader, we provide a complete proof, based on what has been sketched by M. Emerton in \cite{emerton}. 
        
        Choose a finite presentation for the module $N$: $$R^s \to R^t\to N\to 0.$$
		For each $i$, tensor the above presentation with $M_i$ to get an exact sequence of $R$--modules:
        $$R^s\otimes_R M_i\to R^t\otimes_R M_i\to N\otimes_R M_i\to 0.$$
		Regarding the first arrow, we define
		  $$K_i\coloneqq\kernel(R^s\otimes_R M_i\to R^t\otimes_R M_i),\qquad C_i\coloneqq\im(R^s\otimes_R M_i\to R^t\otimes_R M_i)$$
		 and break the obtained exact sequence into two short exact sequences:
		$$0\to K_i\to R^s\otimes_R M_i\to C_i\to0,$$
		and
		 $$0\to C_i\to R^t\otimes_R M_i\to N\otimes_R  M_i\to 0.$$
		Note that if $M_i$ is of finite length, then so are the modules $R^s\otimes_R M_i\cong M_i^s$, $R^t\otimes_R M_i\cong M_i^t$, $ K_i$ and $C_i$.
		Consequently, by the above remark, these modules all satisfy the Mittag-Leffler condition. Therefore, taking inverse limit preserves exactness.
		Thus, we have two short exact sequences:
		$$0\to\varprojlim K_i\to\varprojlim (R^s\otimes_R M_i)\to \varprojlim C_i\to0$$ and $$0\to\varprojlim C_i\to\varprojlim (R^t\otimes_R M_i)\to\varprojlim (N\otimes_R M_i)\to0.$$
		Combining these two exact sequences, we obtain an exact sequence
 $$\varprojlim (R^s\otimes_R M_i)\to\varprojlim (R^t\otimes_R M_i)\to\varprojlim (N\otimes_R  M_i)\to0.$$
		On the other hand, tensoring the exact sequence $R^s \to R^t\to N\to 0$ directly with $\varprojlim M_i$ yields an exact sequence $$R^t\otimes_R(\varprojlim M_i)\to R^s\otimes_R(\varprojlim M_i)\to N\otimes_R (\varprojlim\limits_{i}M_i)\to0.$$
		
		We make use of the natural morphisms
		$$R^s\otimes_R(\varprojlim M_i)\to\varprojlim (R^s\otimes_R M_i),\quad
		R^t\otimes_R(\varprojlim M_i)\to\varprojlim (R^t\otimes_R M_i),$$  $$N\otimes_R (\varprojlim\limits_{i}M_i)\to \varprojlim (N\otimes_R  M_i)$$
		to obtain a diagram with two exact rows, whose commutativity can be easily verified via universal properties:
	
		\[
		\begin{tikzcd}
			R^s\otimes_R(\varprojlim M_i) \arrow[r] \arrow[d]& R^t\otimes_R(\varprojlim M_i)\arrow[r] \arrow[d]&N\otimes_R (\varprojlim M_i)\arrow[r]\arrow[d]&0\\
			\varprojlim (R^s\otimes_R M_i)\arrow[r] & \varprojlim (R^t\otimes_R M_i)\arrow[r] &\varprojlim (N\otimes_R  M_i)\arrow[r]&0
		\end{tikzcd}
		\]
		The first two vertical maps are isomorphisms because projective limits commute with finite direct sums. Thus, the last vertical map is also an isomorphism, by universality of cokernels. Note that the third vertical map is the natural map, which does not depend on the choice of the presentation of $N$, hence it is canonical.
	\end{proof}

	In our considerations, most rings will be endowed with some natural topologies. Therefore, it will be beneficial to extend the above lemma to the topological context. Via topological arguments, we are able to obtain some topological variations on the previous Lemma.
	\begin{lemma}[a topological variation of Lemma \ref{EmertonTrick}]
		Let $R$ be a commutative, compact, Hausdorff topological ring. Let $\{M_i\}_i$ be a projective system of finitely generated, compact and Hausdorff topological $R$-modules, such that the transition maps are continuous. Let $N$ be a finitely presented, compact and Hausdorff topological $R$-module. Then, there exists a canonical isomorphism in the category of topological $R$-modules: $$N\otimes_R (\varprojlim\limits_{i}M_i)\xrightarrow{\sim}\varprojlim\limits_{i}(N\otimes_R M_i ) .$$
	\end{lemma}
	\begin{proof}
	Identical to the proof of the previous Lemma, we have the following two short exact sequences $$0\to K_i\to R^s\otimes_R M_i\to C_i\to0, \qquad 0\to C_i\to R^t\otimes_R M_i\to N\otimes_R  M_i\to 0.$$
		Next, we look at each term in these sequences. Note that all maps involved are continuous.
		\begin{itemize}
			\item  Obviously, $R^t\otimes_R M_i\simeq M_i^t$ and $R^s\otimes_R M_i\simeq M_i^s$ are Hausdorff, compact.
			
			\item  Consequently, $C_i={\rm Im}(R^s\otimes_R M_i\to R^t\otimes_R M_i)$ is Hausdorff, compact.
			
			\item 	Consequently, $K_i:={\rm Ker}(R^s\otimes_R M_i\to C_i)$ is Hausdorff, compact.
			
			\item 	It is known (see \cite{folland}) that if $H$ is a closed subgroup of a topological group $G$ (not necessarily Hausdorff), then $G/H$ is Hausdorff. Apply this to the subgroup $C_i$ of $R^t\otimes_R M_i$ to conclude that $N\otimes_R M_i$ is Hausdorff, compact. 
		\end{itemize}
		Thus, all the modules $K_i, R^s\otimes_R M_i, C_i, R^t\otimes_R M_i, N\otimes_R  M_i$ are compact and Hausdorff. By Lemma \ref{compactinverse}, taking projective limit preserves exactness, and the rest of the proof coincides with that of the previous lemma.
	\end{proof}
	\begin{remark}
It is true that if $R$ is a compact, Hausdorff topological ring, then it is profinite. See \cite{profinitering}. However, we do not need this fact in our proof.
	\end{remark}

	\begin{corollary}\label{em3}
		Let $(R,\mathfrak{m})$ be a commutative, Noetherian, local ring, compact in its $\mathfrak{m}$-adic topology. Let $\{M_i\}_i$ be a projective system of finitely generated $R$-modules with continuous transition maps. Let $N$ be another finitely generated $R$-module. Then, there exists a canonical isomorphism$$N\otimes_R (\varprojlim\limits_{i}M_i)\xrightarrow{\sim}\varprojlim\limits_{i}(N\otimes_R M_i ) .$$
	\end{corollary}
	\begin{proof}
			We point out that for $(R, \frak m)$ as above, every finitely generated $R$-module $M$ is finitely presented, compact and Hausdorff in its $\mathfrak{m}$-adic topology. Indeed, the Hausdorff property follows from Corollary 10.20 of \cite{atiyah}. Also, $M$ is compact because there is a surjective continuous homomorphism $A^l\to M$ for some finite $l$. Therefore, the hypotheses of the previous Lemma are satisfied.
	\end{proof}

	\subsection{More Tools from Topology}

    The following topological version of the classical Nakayama Lemma will be very useful in our future considerations.

    \begin{lemma}[topologial Nakayama]\label{top-Nakayama}
Let $R$ be a commutative, topological ring which is Hausdorff ad compact. Let $J$ be a closed $R$--ideal, contained in the Jacobson radical of $R$. Let $M$ be a topological $R$--module, such that $M\cong \varprojlim_i M_i$, where the $M_i$'s are finitely generated, compact, Hausdorff $R$--modules, the transition maps in the projective system $\{M_i\}_{i\in I}$ are surjective and the topology of $M$ is the projective limit topology. Then, the following hold.
\begin{enumerate}
\item If $M/\overline{JM}=0$, then $M=0$.
\item If $M/\overline{JM}$ is finitely generated over $R/J$, then $M$ is finitely generated over $R$.   
\end{enumerate}
    \end{lemma}
\begin{proof}[Proof (sketch)]
Since the maps $\pi_i:M\to M_i$ are surjective and continuous and $M_i$ is finitely generated, Hausdorff, compact, we have
$$M_i=\pi_i(M)=\pi_i(\overline{JM})=\overline{JM_i}=JM_i, \qquad\text{ for all }i\in I.$$
Now, the classical Nakayama lemma applied to the finitely generated $R$--modules $M_i$ implies that $M_i=0$, for all $i$. Consequently, $M=\varprojlim_i M_i=0.$ This concludes the proof of part (1).

For part (2), assume that $M/\overline{JM}$ is generated by the classes $\widehat{x_1}, \dots, \widehat{x_n}$ of elements $x_1, \dots, x_n$ in $M$. Let $N$ be the (compact) $R$--submodule of $M$ generated by the $x_i$'s. Now, apply part (1) to the $R$--module $M/N$, noting that 
$$M/N=\varprojlim_i M_i/N_i, \qquad \text{ where }N_i:=\pi_i(N),$$
for all $i\in I$. This concludes the proof of part (2). 
\end{proof}
    
		\begin{lemma}\label{toptrick}
			Let $\mathcal{G}$ be a first-countable, compact, Hausdorff abelian topological group. Let $Z_1\supseteq Z_2\supseteq\cdots Z_m\supseteq\cdots$ be a descending chain of closed subgroups of $\mathcal G$, such that:
			 $$\bigcap\limits_{n}Z_m=\{0\}.$$
			 Let $W_1\supseteq W_2\supseteq\cdots W_m\supseteq\cdots$ be another descending chain of subgroups of $\mathcal G$, such that $(W_m+Z_m)$ is closed, for all $m$. Then, we have an equality of subgroups of $\mathcal G$:
$$\bigcap\limits_{m}(W_m+Z_m)=\bigcap\limits_{m}\overline{W_m},$$
where $\overline{X}$ means the topological closure of $X$ in $\mathcal G$, for any $X\subseteq \mathcal G$.
		\end{lemma}
		\begin{proof}
			It is easy to see that the right hand side is contained in the left hand side. Let $g$ be an element on the left hand side of the equality in the statement. For every $m$ there exist elements $w_m\in W_m$ and $z_m\in Z_m$ such that $g=w_m+z_m$. Recall that first countable and compact implies sequentially compact, meaning that every sequence has a convergent subsequence. Thus after relabeling if necessary, we may assume that both $\{w_m\}_m$ and $\{z_m\}_m$ are convergent, and the Hausdorff property guarantees that the limit points are well defined. Since $\lim z_m=0,$ we have:
			$$g=\lim_m g_m=\lim_m (w_m+z_m)=\lim_m w_m+\lim_m z_m=\lim_m w_m\in\bigcap_m\overline{W_m}.$$
			This settles the proof.
		\end{proof}
		
	In \S\ref{general-section} and Appendix B we will use the theory of nets. We refer the reader to Chpt. 4 of \cite{folland2} for a complete treatment. For our further use, we cite Theorem 4.29 in loc.cit.:
\begin{proposition}\label{uniquenetlimit}
		Let $X$ be a topological space, then the following are equivalent:
		\begin{enumerate}
		\item $X$ is compact.
		\item Every net in $X$ has a cluster point.
		\item Every net in $X$ has a convergent subnet.
        \end{enumerate}
	\end{proposition}
	
	\subsection{Basic Properties of Fitting Ideals}
	\begin{definition}
		Let $R$ be a commutative ring and $M$ be a finitely generated $R$-module. Choose an exact sequence of $R$--modules
		$$0\to K\to R^t\to M\to0,$$
		 for some $t\in\Bbb N$, and view the elements of $R^t$ as column vectors and $K$ as a submodule of $R^t$. The $r$-th Fitting ideal $\fitzero_R^r(M)$ is defined to be the ideal of $R$ generated by the determinants of all $(t-r)\times (t-r)$ matrices whose columns are in $K$. It is independent of the chosen exact sequence. (See \cite{northcott}.)
	\end{definition}

    The following useful Lemmas are well known and fairly elementary. (See \cite{northcott} for details.)
	\begin{lemma}[projection]\label{surjectiveFitt}
		Suppose $M_1$ and $M_2$ are finitely generated $R$--modules. If there is a surjective morphism $M_1\to M_2$, then we have inclusions of ideals
		$$\fitzero_{R}^{r}(M_1)\subseteq\fitzero_{R}^{r}(M_2),\qquad \text{ for all } r\geq 0.$$
	\end{lemma}
	
	\begin{lemma}[base--change]\label{fitt-tensor}
		Let $R'$ be a commutative $R$-algebra. Then, for all $r\geqslant0$, we have an equality of $R'$--ideals
		$$\fitzero_{R'}^r(M\otimes_{R}R')=\fitzero_{R}^r(M)_{R'},$$
		where the right-hand side is the ideal of $R'$ generated by $\fitzero_{R}^r(M)$ by extension of scalars along the structural morphisms $R\to R'$.
	\end{lemma}

    Throughout the paper, we will use the notation $I_{R'}$ to denote the ideal generated by $I\subseteq R$ in $R'$ by extension of scalars along a given morphism of commutative rings $R\to R'$.
	
 	\subsection{The Theorem of Greither--Kurihara}
	We will need a couple of additional notations. Let
	$\Lambda=\mathbb{Z}_p\llbracket T\rrbracket$,
	$R=\mathbb{Z}_p[G]\llbracket T\rrbracket$, where $G$ is a finite abelian $p$-group. Note that the rings $\Lambda$ and $R$ are local, of maximal ideals $\frak m=(p, T)$ and $\frak m_R=(p, T, (\sigma-1)_{\sigma\in G})$, respectively and that they are compact and Hausdorff in the $\frak m$--adic topology.
    
	Let $\omega_n$ be the usual Weierstrass polynomial $\omega_n=(1+T)^{p^n}-1$ and set $R_n=R/\langle\omega_n\rangle$. It is well-known that there is a topological isomorphism $R\cong\varprojlim R_n$.
    
Further, we let $A_n$ be a finitely generated $R_n$-module, for all $n\geq 1$, such that $(A_n)_n$ forms a projective system in the category of $R$-modules. Let $X\coloneqq\varprojlim A_n$ be the projective limit, viewed as an $R$--module.
	In this context, Greither and Kurihara proved the following result.
	\begin{theorem}[\cite{grku1}, Theorem 2.1]\label{grku}
		Assume that the projective system $(A_n)_n$ satisfies the following two properties:
		
		(i) $\{A_n\}_n$ is surjective from some $n_0\in \mathbb{N}$ onwards, meaning that all the transition maps $A_{n+1}\to A_n$ are surjective for indices $n\geqslant n_0$.
		
		(ii) The limit $X=\varprojlim A_n$ is a finitely generated, torsion module over $\Lambda$.
		
		If $\iota$ denotes the natural identification $\iota: R\xrightarrow{\sim}\varprojlim R_n$, then
	$$\iota(\fitzero^0_R(X))=\varprojlim\limits_{n}\fitzero^0_{R_n}(A_n).$$
	\end{theorem}
	The goal of this paper is to formulate and prove various generalizations of Theorem \ref{grku}. Before doing that, we need to note that if $G$ is a finite abelian group with $G=P\times \Delta$, where $P$ is the Sylow $p$--subgroup of $G$, then the algebra $\Bbb Z_p[G]$ decomposes as a direct sum
    \begin{equation}\label{dec-groupring-eq}\Bbb Z_p[G]\cong \bigoplus_{[\chi]\in\widehat{\Delta}(\Bbb Q_p)}\Bbb Z_p(\chi)[P],\end{equation}
    where $\widehat{\Delta}(\Bbb Q_p)$ is the set of equivalence classes of irreducible $\overline{\Bbb Q_p}$--valued characters $\chi$ of $\Delta$, up to conjugation by the action of the absolute Galois group $G(\overline{\Bbb Q_p}/\Bbb Q_p)$. Moreover, the rings $\mathcal O_\chi:=\Bbb Z_p(\chi)[P]$ are local, of maximal ideals $\frak m_\chi:=(p, I_P)$ (here, $I_P$ denotes the augmentation ideal associated to $P$); they are compact and Hausdorff in their $\frak m_\chi$--adic topologies; they are free of finite rank over their common subring $\Bbb Z_p$, which is a PID. So, the Greither--Kurihara Theorem is in fact a statement about Fitting ideals of certain $\mathcal O_\chi[[T]]$--modules, for a ring $\mathcal O_\chi:=\Bbb Z_p(\chi)[P]$, constructed as above. This prompts the following definition.

    \begin{definition}
    A local ring $(\mathcal O, \mathfrak m_{\mathcal O})$ is called $p$--admissible if it satisfies the following.
    \begin{enumerate}
    \item $\mathcal O$ is compact and Hausdorff in its $\mathfrak m_{\mathcal O}$--adic topology.
    \item The residual field $\kappa_{\mathcal O}:=\mathcal O/\mathfrak m _{\mathcal O}$ is (necessarily finite) of characteristic $p$.
    \item $\mathcal O$ is an $A$--algebra and a free $A$--module of finite rank over a PID $A$. (In particular, $\mathcal O$ is Noetherian.)
    \end{enumerate}
    A general commutative ring is called $p$--admissible if it is a finite direct sum of local $p$--admissible rings.
    \end{definition}
    
\begin{remark}\label{examples-p-admissible}The reader can easily check that the following are examples of local, $p$--admissible rings:
    \begin{itemize}
    \item $(S[P], (\pi, I_P))$ where $S$ is a finite extension of $\Bbb Z_p$ of uniformizer $\pi$ and $P$ is a finite (possibly trivial), abelian $p$--group. Here, one can take $A:=\Bbb Z_p$.
\item $(\Bbb F_q[P], I_P)$ and $(F_q[P][[X]], (\pi, I_P, X))$, where $\Bbb F_q$ is a finite field of characteristic $p$ and $P$ is a finite (possibly trivial) abelian $p$--group. Here, $A:=\Bbb F_q$ and $A:=\Bbb F_q[[X]]$, respectively.
\item If $(\mathcal O, \frak m_{\mathcal O})\subseteq (\mathcal O', \frak m_{\mathcal O'})$ is an extension of local rings (meaning $\frak m_{\mathcal O'}\cap\mathcal O=\frak m_{\mathcal O}$), such that $\mathcal O$ is $p$--admissible and $\mathcal O'$ is a free $\mathcal O$--module of finite rank, then $\mathcal O'$ is $p$--admissible. If $A$ works for $\mathcal O$, it works for $\mathcal O'$.
\end{itemize} 
If $\mathcal O$ is a local, $p$--admissible ring and $G$ is a finite group, then $\mathcal O[G]$ is a (semilocal) $p$--admissible ring. The argument mimics the direct sum decomposition \eqref{dec-groupring-eq}.
\end{remark}    
    \begin{remark} One last remark worth making is that although we will only focus on the $0$--th Fitting ideals of the modules in question, all the theorems which follow hold for all the higher Fitting ideals $\fitzero^r(-)$, with $r\geq 1$, as well.
	\end{remark}
    
    \section{Removing the ``torsion'' hypothesis. General $p$--admissible coefficients}
	
	We fix a local $p$--admissible ring $(\mathcal O, \frak m_{\mathcal O})$ and let $R:=\mathcal O[[T]]$. Note that $R$ is Noetherian, local, of maximal ideal $\frak m_R=(\frak m_{\mathcal O}, T)$, compact and Hausdorff in its $\frak m_R$--adic topology. Also, the Weierstrass preparation theorem holds for $R$ (see \cite{Bourbaki_comm_alg}, Chapter 7, Section 3) and, as a consequence, the $\mathcal O$--algebras
    $R_n:=R/(\omega_n(T))$, where $\omega_n(T)=(1+T)^{p^n}-1$ are the usual Weierstrass polynomials, are free $\mathcal O$--modules of rank $p^n$. Moreover, since $\omega_n(T)\in\frak m_R^{n+1}$, for all $n$, and therefore $\bigcap\limits_n\langle \omega_n(T)\rangle =\langle  0\rangle$ by Hausdorff-ness, we have a natural isomorphism of topological rings
    $$R\cong\varprojlim_n R_n,$$
    where the left side is endowed with the $\frak m_R$--adic topology and the right side is endowed with the projective limit of the $\frak m_O$--adic topologies.

    In this context, we consider a projective system $(A_n)_n$ of $R$--modules, such that each $A_n$ is also an $R_n$--module, and take its projective limit $X:=\varprojlim_n A_n$ in the category of $R$--modules. In this section, we prove the following mild generalization of Theorem \ref{grku}, which is the first step towards the higher (ultimately infinite) rank generalizations to follow.
    
    \begin{theorem}\label{theorem0}
		Assume that the projective system $(A_n)_n$ satisfies the following two properties:
		
		(i) $(A_n)_n$ is surjective from some $n_0\in \mathbb{N}$ onwards.
		
		(ii) The limit $X$ is a finitely generated module $R$--module.
		
		If $\iota$ denotes the natural identification $\iota: R\xrightarrow{\sim}\varprojlim\limits_{n} R_n$, then
		$$\iota(\fitzero^0_R(X))=\varprojlim\limits_{n}\fitzero^0_{R_n}(A_n).$$
	\end{theorem}

    \begin{remark} Note that the Theorem above implies (by simply taking a direct sum) a similar result for rings of the type $B[[T]]$, where $B$ is a general (not necessarily local) $p$--admissible ring, for example $B:=\Bbb Z_p[G]$, where $G$ is any finite, abelian group. (See Remark \ref{examples-p-admissible}.) we state this formally as a Corollary below.
    \end{remark}

    \begin{corollary}\label{cor-theorem0} The statement of Theorem \ref{theorem0} holds if one replaces the ring of coefficients $\mathcal O$ by a general (not necessarily local) $p$--admissible ring.
	\end{corollary}
	Before we begin, note that discarding a finite number of terms would not change the inverse limit above, so we may, and shall always assume that the system $(A_n)_n$ is surjective.
The proof of Theorem \ref{theorem0} follows very closely the ideas of Greither and Kurihara. However, we give it in detail, hoping that this will make the next sections of the paper more readable. The proof requires some additional notations and several lemmas. For simplicity, we let
$$\mathcal F:=\fitzero_R^0(X).$$
Next, we define the following $R_m$--modules, for all $n, m$.
	$$A_{n,m}\coloneqq A_n/\omega_mA_n=A_n\otimes_R\big(R/\omega_mR\big).$$
	Notice that for	$m\geqslant n$, since $\omega_n\mid \omega_m$, we have 
	$$\langle\omega_m\rangle\subseteq\langle\omega_n\rangle,\qquad\omega_mA_n=0,\qquad \text{ for all }m\geqslant n.$$ This further implies $A_{n,m}=A_n$, for all $m\geq n$.
	Thus, for any fixed $n$, the projective system $\{A_{n,m}\}_m$ is stationary for $m$ large enough. Thus, we have equalities 
    $$\varprojlim\limits_{m}A_{n,m}=A_n, \qquad \varprojlim\limits_{n}\varprojlim\limits_{m}A_{n,m}=\varprojlim\limits_{n}A_n=X.$$
	By Lemma \ref{commutativity} in the Appendix, we have a canonical isomorphism of $R$-modules
$$\varprojlim\limits_{n}\varprojlim\limits_{m}A_{n,m}\cong\varprojlim\limits_{m}\varprojlim\limits_{n}A_{n,m}.$$
	Define the $R_m$--module $X_m\coloneqq\varprojlim\limits_{n}A_{n,m}$ and let $E_m\coloneqq\fitzero_{R_m}^0(X_m)$.  So, we have $$X\cong\varprojlim\limits_m X_m.$$
    Note that since the $R$--module $X$ is finitely generated and the natural maps $X\to A_n$, $X\to A_{n, m}$ and $X\to X_m$ are surjective, the $R$--modules $A_n$, $A_{n, m}$, $X_m$ are finitely generated. Since $R$ is Noetherian, compact and Hausdorff, the $R$--Fitting ideals of all these modules are finitely presented, compact and Hausdorff. 
	\begin{lemma}\label{lemma2.1}
		The map $\iota$ induces an isomorphism
		$$\mathscr{F}\cong\varprojlim\limits_{m}E_m.$$
	\end{lemma}
	\begin{proof}
		We have natural isomorphisms of $R_m$--modules:
		\begin{align*}
			X_m\coloneqq&\varprojlim\limits_{n}A_{n,m}=\varprojlim\limits_{n}\Big(A_n\otimes_R\big(R/\omega_mR\big)\Big)\\
			\cong&\Big(\varprojlim\limits_{n}A_n\Big)\otimes_R\big(R/\omega_mR\big)\\
			=&X\otimes_R\big(R/\omega_mR\big)\cong X/\omega_mX.
		\end{align*}
		Note that the second-line isomorphism above follows from Corollary \ref{em3}. Therefore, Lemma \ref{fitt-tensor} gives the following equalities of $R_m$--ideals: 
		\begin{align*}
			E_m=\fitzero_{R_m}^0(X\otimes_R R_m)
			=\fitzero_{R}^0(X)_{R_{m}}
=\mathscr{F}_{R_{m}}=\Big(\mathscr{F}+\langle\omega_m\rangle\Big)/\langle\omega_m\rangle.
        \end{align*}
	Consequently, we have the following projective system of short exact sequences of finitely generated, therefore compact and Hausdorff, topological $R$--modules
$$0\to\langle\omega_m\rangle\to\Big(\mathscr{F}+\langle\omega_m\rangle\Big)\to E_m\to0,$$
whose transition maps are induced by the natural inclusions $\langle \omega_{m'}(T)\rangle \subseteq \langle \omega_{m}(T)\rangle$, for all $m'\geqslant m.$
Therefore, when we take the projective limit,  we do not  lose exactness:
$$0\to\varprojlim\limits_{m}\,\langle\omega_m\rangle\to\varprojlim\limits_{m}\Big(\mathscr{F}+\langle\omega_m\rangle\Big)\to \varprojlim\limits_{m}E_m\to0.$$
		Notice that  we have $$\varprojlim\limits_{m}\,\langle\omega_m\rangle=\bigcap\limits_{m}\langle\omega_m\rangle=0.$$
		For the second term in the exact sequence, we apply Lemma \ref{toptrick} to the topological group $\mathcal{G}=(R, \, +)$, and closed subgroups $Z_m=\langle\omega_m\rangle$ and $W_m=\mathscr{F}$ to obtain an equality of ideals:
		$$\varprojlim\limits_{m}\Big(\mathscr{F}+\langle\omega_m\rangle\Big)=\bigcap\limits_{m}\Big(\mathscr{F}+\langle\omega_m\rangle\Big)=\mathscr{F}.$$
		Consequently, the above exact sequence reads
	$$0\to0\to\mathscr{F}\to\varprojlim\limits_{m}E_m\to0,$$
	which settles the proof of Lemma \ref{lemma2.1}.
	\end{proof}
	
	The next lemma is essentially a mild modification of Theorem \ref{grku} in \cite{grku1}.

	\begin{lemma}\label{lemma2.2}
		Let $A$ be a PID and let $(S,\mathfrak{m})$ be a local $A$-algebra, which is compact and Hausdorff in its $\frak m$--adic topology. Assume further that $S$ is a free $A$-module of finite rank.
		Let $\{B_{n}\}_{n}$ be a projective system of $S$-modules, and let $B\coloneqq\varprojlim\limits_{n}B_{n}$. Assume that
\begin{itemize}		
		\item[(i)] The projective system of modules $\{B_{n}\}_{n}$ is surjective.
		\item[(ii)] The projective limit $B\coloneqq\varprojlim B_n$ is a finitely generated  $S$--module.
		\end{itemize}
		Then, the natural inclusions $\fitzero_{S}^0(B_{n+1})\subseteq\fitzero_{S}^0(B_n)$ induce equalities of $S$--ideals $$\fitzero_{S}^0(B)=\bigcap\limits_{n}\fitzero_{S}^0(B_n)=\varprojlim\limits_{n}\fitzero_{S}^0(B_n).$$
	\end{lemma}
	\begin{proof}
	    Note that $S$ is a Noetherian ring, as a finitely generated module over the Noetherian ring $A$.
Since the maps $B\to B_n$ are surjective and $B$ is finitely generated, $B_n$ is finitely generated as an $S$--module, for all $n$. Let us denote by $d_n$ the cardinality of a minimal set of generators of $B_n$ over $S$. If we let $\kappa:=S/\frak m$, by Nakayama's Lemma, we must have $$d_n=\dim_{\kappa}(B_n/\mathfrak{m}B_n).$$ Since $B_{n+1}\twoheadrightarrow B_n$ is surjective, one has $d_{n+1}\geqslant d_n$. On the other hand, set $d=\dim_{\kappa}(B/\mathfrak{m}B)$. Since $B\to B_n$ is surjective, we have $d\geqslant d_n$, for all $n$. Thus, the sequence $(d_n)_n$ is an increasing, bounded sequence of integers, hence stationary. Let $N\in\Bbb N$, such that $d_{N}=d_{N+1}=\cdots$. For simplicity, let $t:=d_{N}$.
		
		Now, via an iterative process, we are going to produce a ``coherent'' system of generators for the modules $(B_n)_n$ for all  $n\geqslant N$. For $B_N$, we choose a set of generators $b_1^{(N)},\cdots, b_t^{(N)}$. Then, since $B_{N+1}\to B_N$ is surjective, there exist elements $b_1^{(N+1)},\cdots, b_t^{(N+1)}$ such that $b_i^{(N+1)}\mapsto b_i^{(N)}$, for all $i$. Note that $$B_{N+1}/\mathfrak{m}B_{N+1}\cong B_{N}/\mathfrak{m}B_{N},$$ so the images of $\{b_i^{(N+1)}\}_i$ generate $B_{N+1}/\mathfrak{m}B_{N+1}$, and hence $\{b_i^{(N+1)}\}_i$ generates $B_{N+1}$ by Nakayama's Lemma. Repeating this procedure, we obtain a compatible system of generators, which can be expressed via the following commutative diagram (for each $n\geqslant N$):
		\[
		\begin{tikzcd}
			0\arrow[r]&K_{n+1}\arrow[d]\arrow[r]&S^t\arrow[d, equal]\arrow[r]&B_{n+1}\arrow[d]\arrow[r]&0\\
			0\arrow[r]&K_{n}\arrow[r]&S^t\arrow[r]&B_{n}\arrow[r]&0
		\end{tikzcd}\]
	where $K_n$ denotes the kernel $\operatorname{Ker}(S^t\to B_n)$. Since $A$ is a PID and $S^t$ is a  free $A$-module of finite rank, so is $K_n$. We apply the Snake Lemma to this diagram, and see that $K_{n+1}\to K_{n}$ is injective. Hence, if we let $h:={\rm rank}_A(K_N)$, we have: $$\operatorname{rank}_A(K_{n+1})\leqslant\operatorname{rank}_A(K_{n})\leqslant h\leqslant \operatorname{rank}_A(S)\cdot t,\qquad\text{ for all }n\geq N.$$
	 Then, every $K_n$ can be generated by at most $h$ elements over $A$, hence \textit{a fortiori}, by at most $h$ elements over $S$. Also, by taking projective limits in the diagram above (and noting that all $S$--modules involved are compact and Hausdorff in their $\frak m$--adic topology and all maps between them are continuous) we obtain an exact sequence of $S$--modules:
\[\begin{tikzcd}
0\arrow[r]&\bigcap\limits_{n}K_{n}\arrow[r]&S^t\arrow[r]&B\arrow[r]&0.
	\end{tikzcd}\]
	Now, we are ready to prove the equality $\fitzero_{S}^0(B)=\bigcap\limits_{n}\fitzero_{S}^0(B_n).$ The left ideal is included in the right as an immediate consequence of the surjective maps $B\to B_n.$ So, we will concern ourselves with proving the opposite inclusion $\bigcap\limits_{n}\fitzero_{S}^0(B_n)\subseteq\fitzero_{S}^0(B).$ If the intersection is trivial, we are done. So, assume that this is not the case and let $$\beta\in\bigcap\limits_{n}\fitzero_{S}^0(B_n), \qquad \beta\ne 0.$$
	 We want to show that $\beta\in\fitzero_{S}^0(B).$
	 By the choice of $\beta$, for each $n\geqslant N$, there exist at most $h'={h \choose t}$ matrices $Y_{n}^{(i)}, i\in\{1,2,\cdots h'\}$, such that $$\beta=\sum\limits_{i=1}^{h'}\det(Y_{n}^{(i)}),$$
	where for each $i$, $Y_{n}^{(i)}$ is a $t\times t$ matrix with rows in $K_n$. Since $S$ is compact, there exists a subsequence $n_1< n_2< n_3\cdots$ such that for each $i$, $\{Y^{(i)}_{n_j}\}_j$ converges in $M_{t\times t}(S).$
	For each row $\vec{y}^{(i)}_{n_j}$ of $Y^{(i)}_{n_j}$, we have $\lim\limits_{j\to\infty}\vec{y}^{(i)}_{n_j}\in K_{n_{l}}$ for any $l$, since $K_n$ is closed for any $n$. Hence:
	$$\lim\limits_{j\to\infty}\vec{y}^{(i)}_{n_j}\in \bigcap\limits_{l}K_{n_{l}}=\bigcap\limits_{n}K_n.$$
If combined with the last exact sequence, this leads to the conclusion that 
$\lim\limits_{j\to\infty}Y^{(i)}_{n_j}$ is a $t\times t$ matrix with rows in $\operatorname{Ker}(S^t\to B)$.
Now, since determinants are continuous functions in the matrix entries, we obtain:
	\begin{align*}
		\beta=\lim\limits_{j}\beta
		=&\lim\limits_{j}\sum\limits_{i=1}^{h'}\det(Y_{n_j}^{(i)})\\
		=&\sum\limits_{i=1}^{h'}\lim\limits_{j}\det(Y_{n_j}^{(i)})
		=\sum\limits_{i=1}^{h'}\det(\lim\limits_{j}Y_{n_j}^{(i)})
		\in\fitzero_S^0(B).
	\end{align*}
	This settles the proof of the Lemma.
	\end{proof}
	
	Next, in the context of Theorem \ref{theorem0}, we introduce the additional notations $\mathscr{F}_n\coloneqq\fitzero_R^0(A_n)$ and $F_n\coloneqq\fitzero_{R_n}^0(A_n)$. Then, the statement of the theorem can be rewritten as $$\iota: \mathscr{F}\cong\varprojlim\limits_{n}F_n.$$
	The following Lemma shows that it is sufficient to prove $\mathscr{F}=\bigcap\limits_{n}\mathscr{F}_n$ instead.
	\begin{lemma}\label{lemma2.3}
		The following holds:
		$$\varprojlim\limits_{n}\mathscr{F}_n=\bigcap\limits_{n}\mathscr{F}_n, \qquad \iota: \bigcap\limits_{n}\mathscr{F}_n\cong\varprojlim\limits_{n}F_n.$$
	\end{lemma}
	\begin{proof}
		Since the map $A_{n+1}\to A_n$ is surjective, one has natural inclusion $\mathcal F_{n+1}\subseteq\mathcal  F_n$, due to Lemma \ref{surjectiveFitt}. Thus, the equality in the statement is straightforward. Now, we note that there is an isomorphism of $R_n$--modules $A_n\otimes_R R_n\cong A_n.$ Whence, we have
		\begin{align*}
			F_n\coloneqq&\fitzero_{R_n}^0(A_n)
			=\fitzero_{R_n}^0(A_n\otimes_R R_n)
			=\fitzero_R^0(A_n)_{R_n}
			\cong\Big(\mathscr{F}_n+\langle\omega_n\rangle\Big)/\langle\omega_n\rangle.
		\end{align*}
Thus, for all $n$, we have a projective system of short exact sequences of finitely generated (therefore compact and Hausdorff) $R$-modules, with transition maps induced by the inclusions $\langle \omega_{n'}(T)\rangle\subseteq \langle \omega_n(T)\rangle$, for all $n'\geqslant n$. 
$$0\to\langle\omega_n\rangle\to\Big(\mathscr{F}_n+\langle\omega_n\rangle\Big)\to F_n\to 0.$$
	Taking a projective limit with respect to $n$, we get an exact sequence:
$$0\to\bigcap\limits_{n}\langle\omega_n\rangle\to\bigcap\limits_{n}\Big(\mathscr{F}_n+\langle\omega_n\rangle\Big)\to\varprojlim\limits_{n} F_n\to0.$$
Now, applying Lemma \ref{toptrick} to the middle intersection gives $\bigcap\limits_{n}\langle\omega_n\rangle=0$, we get $$\bigcap\limits_{n}\Big(\mathscr{F}_n+\langle\omega_n\rangle\Big)=\bigcap\limits_{n}\mathscr{F}_n.$$
	This, combined with $\cap_n\langle \omega_n(T)\rangle=\langle 0\rangle$, settles the proof of the Lemma.
	\end{proof}
	\medskip
    
	\begin{proof}[Proof of Theorem \ref{theorem0}] 
    
For each $m$, we are planning on applying Lemma \ref{lemma2.2} to the data:
$$S=R_m,\quad B_n:=A_n\otimes_R R/\omega_mR=A_{n,m},\quad B:=\varprojlim_n{B_n}=X_m\cong X/\omega_mX.$$
Note that $S$ is indeed local, compact and Hausdorff in its maximal ideal topology. Since $S$ is a  free $\mathcal O$--module of rank $p^n$, and $\mathcal O$, by definition, is a free module of finite rank over a PID $A$, then $S$ is a free module of finite rank over the same $A$. Further, note that $B$ is a finitely generated $S$-module and the projective system $\{B_n\}_n$ is surjective. Thus, Lemma \ref{lemma2.2} gives a canonical isomorphism
$$\fitzero_{S}^0(B)\cong\varprojlim\limits_{n}\fitzero_{S}^0(B_n).$$
When combining this isomorphism with Lemma \ref{lemma2.1} (note that $B=X_m$ in the notations of loc.cit.), we obtain natural isomorphisms
$$\mathscr{F}\cong\varprojlim\limits_{m}\fitzero_{R_m}^0(\varprojlim\limits_{n}A_{n,m})\cong\varprojlim\limits_{m}\varprojlim\limits_{n}\fitzero_{R_m}^0(A_{n,m}).$$
Consequently, Lemma \ref{commutativity} and Lemma \ref{lemma2.3} give natural isomorphisms
\begin{align*}
\mathscr{F}\cong \varprojlim\limits_{n}\varprojlim\limits_{m}\fitzero^0_{R}(A_{n,m})&=\varprojlim\limits_{n}\varprojlim\limits_{m}\fitzero^0_{R}(A_n)_{R_m}\\
&\cong\varprojlim\limits_{n}\varprojlim\limits_{m}\Big(\big(\mathscr{F}_n+\langle\omega_m\rangle\big)/\langle\omega_m\rangle\Big)
		\cong\varprojlim\limits_{n}\mathscr{F}_n
		\cong\varprojlim\limits_{n}{F}_n.  
\end{align*}
This settles the proof of Theorem \ref{theorem0}.
\end{proof}

	\section{The ``Finite Rank'' Case}
	As in the previous section, we fix a local $p$--admissible ring $(\mathcal O, \frak m_O)$. We let 
		$$R\coloneqq\mathcal O\llbracket T_1,T_2,\cdots,T_s\rrbracket.$$  It is well-known that $R$ is a Noetherian, local ring of maximal ideal $\frak m_R=\langle \frak m_{\mathcal O}, T_1, \dots, T_s\rangle$, compact and Hausdorff in its $\frak m_R$--adic topology. 
        
        It is well known that if one considers the profinite group $\Gamma=\mathbb{Z}^{s}_p$ (so, a free $\Bbb Z_p$--module of rank $s$), then its profinite group algebra  $\mathcal O\llbracket\Gamma\rrbracket$ is isomorphic to $R$, as topological $\mathcal O$--algebras. This explains  why we are using the terminology ``finite rank $s$.''  From this perspective, Theorems \ref{grku} and \ref{theorem0} can be regarded as ``rank $1$'' instances of a more general result to be described below.
        
		Further, we let $\omega_m(T)=(T+1)^{p^m}-1$ be the usual Weierstrass polynomial, and let $$\Omega_m=\langle\omega_m(T_1),\omega_m(T_2),\cdots,\omega_m(T_s)\rangle_R,\qquad R_m\coloneqq R/\Omega_m.$$  Since $\Omega_m\subseteq \frak m_R^{m+1}$, for all $m$,  there is a canonical topological $\mathcal O$--algebra isomorphism $$\iota : R\cong\varprojlim\limits_{m} R_m,$$
        where the right side is endowed with the projective limit of the corresponding $\frak m_{\mathcal O}$--adic topologies of the rings $R_m$. Note that the Weierstrass preparation theorem implies that each $R_m$ is a free $\mathcal O$--module of rank $(p^m)^s$. Its $\frak m_{\mathcal O}$--adic and $\frak m_{R_m}$--adic topologies are identical.
	
		In this context, we take a projective system $(A_m)_m$ of $R_m$--modules, whose transition maps are $R$--module morphisms. Let $X\coloneqq\varprojlim_{m} A_m$ be its projective limit in the category of $R$--modules. The main goal of this section is the proof of the following result, essentially known to the authors of \cite{grku1} in the particular case $\mathcal O:=\Bbb Z_p[G]$, with $G$ a finite, abelian $p$--group, and under a certain torsion hypothesis, as stated without proof in loc.cit.

		\begin{theorem}\label{theorem1}
			Assume that the projective system $(A_m)_m$ satisfies the following:
			\begin{enumerate}

			\item[(i)]  $(A_m)_m$ is surjective, in the usual sense.
			
			\item[(ii)] The projective limit $X\coloneqq\varprojlim\limits_{m}A_{m}$ is a finitely generated $R$--module.
            \end{enumerate}
			Then, we have an equality of $R$--ideals
			$$\iota(\fitzero^0_R(X))=\varprojlim\limits_{m}\fitzero^0_{R_m}(A_m).$$
		\end{theorem}
As in the previous section, we record the obvious corollary.
\begin{corollary}\label{corollary-theorem1}
The statement of Theorem \ref{theorem1} holds true if one replaces the ring of coefficients $\mathcal O$ with an arbitrary (not necessarily local) $p$--admissible ring.
\end{corollary}
\medskip

The proof of the theorem above will proceed by induction on $s$, with Theorem \ref{theorem0} as the initial step,  and requires a couple of technical Lemmas.  For simplicity, we let:
$$\mathscr{F}\coloneqq\fitzero^0_R(X), \qquad
			\mathscr{R}_k\coloneqq R/\langle\omega_k(T_s)\rangle,
			\qquad A_{m,k}\coloneqq A_m\otimes_R \mathscr{R}_k,$$
$$X_k\coloneqq \varprojlim\limits_{m} A_{m,k}, \qquad F_k\coloneqq\fitzero^0_{\mathscr{R}_k}(X_k).$$
			Note that $\mathscr{F}$ is an ideal of $R$ and $F_k$ is an ideal of $\mathscr{R}_k$.
		\begin{lemma}\label{lemma 3.1}
			The topological ring isomorphism $\iota_s: R\cong\varprojlim_k\mathscr R_k$ induces an isomorphism
			$$\iota_s: \mathscr{F}\cong\varprojlim\limits_{k} F_k$$  
		\end{lemma}\label{lemma3.1}
		\begin{proof}
			First, we apply Corollary \ref{em3} to compute
			\begin{align*}
				X_k &=\varprojlim\limits_{m} A_{m,k}=\varprojlim\limits_{m}(A_m\otimes_R \mathscr{R}_k)
				=\varprojlim\limits_{m} \Big(A_m\otimes_R \big(R/\langle\omega_k(T_s)\rangle\big)\Big)\\
				&\cong(\varprojlim\limits_{m}A_m)\otimes_R \Big(R/\langle\omega_k(T_s)\rangle\Big)
=X\otimes_R \Big(R/\langle\omega_k(T_s)\rangle\Big)
				=X /\omega_k(T_s)X.
			\end{align*}
			Thus, by properties of Fitting ideals, we have the following: $$\fitzero^0_{\mathscr{R}_k}(X_k)=\fitzero_R^0(X)_{\mathscr{R}_k}\cong\Big(\mathscr{F}+\langle\omega_k(T_s)\rangle\Big)/\langle\omega_k(T_s)\rangle.$$
Consequently, we have a projective system of short exact sequences of $R$-modules:
			$$0\to\langle\omega_k(T_s)\rangle\to \mathscr{F}+\langle\omega_k(T_s)\rangle\to F_k\to 0.$$
			whose transition maps are induced by the inclusions $\langle\omega_{k'}(T_s)\rangle\subseteq\langle\omega_{k}(T_s)\rangle$, for all $k'\geqslant k$.  Note that these are all finitely generated $R$-modules, hence compact, as $R$ is compact and Noetherian. Thus, the projective limit of the above sequences preserves exactness:
	$$0\to\varprojlim\limits_{k}\, \langle\omega_k(T_s)\rangle\to\varprojlim\limits_{k}\Big(\mathscr{F}+\langle\omega_k(T_s)\rangle\Big)\to \varprojlim\limits_{k} F_k\to0.$$
			Noting that $\langle\omega_k(T_s)\rangle_k$ is a descending chain of ideals of $R$, we have:
$$\varprojlim\limits_{k}\,\langle\omega_k(T_s)\rangle=\bigcap\limits_{k}\langle\omega_k(T_s)\rangle=\langle 0\rangle , \qquad \varprojlim\limits_{k}\Big(\mathscr{F}+\langle\omega_k(T_s)\rangle\Big)=\bigcap\limits_{k}\Big(\mathscr{F}+\langle\omega_k(T_s)\rangle\Big)=\mathscr F,$$
where the first equality is a consequence of the fact that $R$ is Hausdorff	and the second is a direct consequence of Lemma \ref{toptrick}. When combined with the last exact sequence, these equalities settle the proof of the Lemma.
		\end{proof}
\medskip

		Next, we have a closer look at the topological rings $\mathscr{R}_k$. We start by noting that, as a consequence of the Weierstrass preparation theorem, the local ring 
        $$\mathcal O'_k:=\mathcal O\llbracket T_s\rrbracket/\langle\omega_k(T_s)\rangle$$
        is a local ring  extension of $\mathcal O$, which is a free $\mathcal O$--module of rank $p^k$. 
        By Remark \ref{examples-p-admissible}, $\mathcal O'_k$ is itself a $p$--admissible local ring. Consequently, since we can write
        $$\mathscr{R}_k
			=\mathcal O\llbracket T_1,T_2,\cdots,T_s\rrbracket/\langle\omega_k(T_s)\rangle\cong\big(\mathcal O\llbracket
             T_s\rrbracket/\langle\omega_k(T_s)\rangle\big)\llbracket T_1,\cdots T_{s-1}\rrbracket=\mathcal O'_k\llbracket T_1, \dots, T_{s-1}\rrbracket,$$
        the topological ring $\mathscr R_k$ is an Iwasawa algebra of rank $(s-1)$ with coefficients in a $p$--admissible local ring.     
        This observation will permit us to prove Theorem \ref{theorem1} by induction on $s$.

\medskip		

\begin{proof}[Proof of  Theorem \ref{theorem1}] The proof will proceed by induction on $s$, with everything else in the statement (the coefficient ring $\mathcal O$, the projective system $\{A_m\}_m$ satisfying the required hypotheses etc.) being arbitrary. The base case $s=1$ is Theorem \ref{theorem0}. Assume that the statement holds for $(s-1)$ and, for a fixed $k$,  apply this hypothesis to the following data:
		\begin{itemize}
			\item The rank $(s-1)$ Iwasawa algebra $\mathscr{R}_k=\mathcal O'_k\llbracket T_1, \dots, T_{s-1}\rrbracket$.
			\item The surjective projective system $\{A_{m,k}\cong A_m\otimes_R \mathscr{R}_k\}_m$ of modules over
            $$R_{m,k}\coloneqq R/\Big(\Omega_m+\omega_k(T_s)\Big)\cong \mathscr{R}_k/{\Omega_m}\mathscr R_k.$$
			\item The projective limit $X_k:=\varprojlim\limits_{m}A_{m,k}$, which is finitely
			generated over $\mathscr R_k$, as $X$ is finitely generated over $R$ and there is an $R$--module isomorphism
            $$X_k\simeq X/\omega_k(T_s)X. $$
		\end{itemize}
By the induction hypothesis, we have equalities
		\begin{align*}
			F_k=&\fitzero^0_{\mathscr{R}_k}(X_k)
			=\fitzero^0_{\mathscr{R}_k}(\varprojlim\limits_{m}A_{m,k})
			=\varprojlim\limits_{m}\fitzero^0_{\mathscr{R}_k}(A_{m,k}).
		\end{align*}
		Hence, by Lemma \ref{lemma 3.1} and Lemma \ref{commutativity}, we have the following:
			\begin{align*}
			\mathscr{F}&\cong\varprojlim\limits_{k}F_k
			  \cong\varprojlim\limits_{k}\varprojlim\limits_{m}\fitzero^0_{R_{m,k}}(A_{m,k})
			  \cong \varprojlim\limits_{m}\varprojlim\limits_{k}\fitzero^0_{R_{m,k}}(A_{m,k}),
		\end{align*}
		Thus, in order to prove the theorem it suffices to show that there is a natural isomorphism:
		\begin{gather}\label{finiterank1}
				\varprojlim\limits_{m}\varprojlim\limits_{k}\fitzero^0_{R_{m,k}}(A_{m,k})\cong\varprojlim\limits_{m}\fitzero^0_{R_m}(A_m).		
		\end{gather}
		 Remarking that for $k\geq m$, we have $\omega_k(T_s)\in\Omega_m$, we conclude that
		$$R_{m,k}\coloneqq R/\Big(\Omega_m+\langle\omega_k(T_s)\rangle\Big)=R/\Omega_m=R_m, \qquad \text{for all }k\geq m.$$
	We have a similar equality $A_{m,k}=A_m$, for all $k\geq m$. Hence, we have:
								$$\varprojlim\limits_{k} R_{m,k}=R_m, \qquad \varprojlim\limits_{k} A_{m,k}=A_m. $$
		Therefore, at the level of Fitting ideals, we have
		$$\fitzero^0_{R_{m,k}}(A_{m,k})=\fitzero^0_{R_{m}}(A_{m}),\text{ for $k\geq m$},  \qquad
		\varprojlim\limits_{k} \fitzero^0_{R_{m,k}}(A_{m,k})=\fitzero^0_{R_{m}}(A_{m}).$$
		When taking limits with respect to $m$ in the last equality above, we obtain the desired equality \eqref{finiterank1},	
		which concludes the proof of the Theorem.
\end{proof}
	
\section{The ``Well-Structured'' Infinite Rank Case}

Next, we would like to generalize Theorem \ref{theorem1} to the ``infinite rank'' case. We start, as usual, with a local $p$--admissible ring $(\mathcal O, \frak m_{\mathcal O})$ and consider the $\mathcal O$--algebra of power series in countably many variables defined as 
$$R:=\mathcal O\llbracket T_1, T_2,\dots\rrbracket:=\varprojlim_n\mathcal O\llbracket T_1, \dots, T_n\rrbracket,$$
where the projective limit is taken with respect to the usual surjective transition maps sending the extra variable to $0$ and keeping the rest intact. We endow $R$ with the projective limit of the $\frak m_{R_n}$--topologies on the local rings $$(R_n:=O\llbracket T_1, \dots, T_n\rrbracket,\, \frak m_{R_n}:=(\frak m_O, T_1, \dots, T_n)).$$ As a projective limit of compact, Hausdorff spaces, $R$ is compact and Hausdorff. It is easily seen that $R$ is a local ring, of maximal ideal
$$\frak m_R:=\varprojlim_n\, \frak m_{R_n}=\overline{(\frak m_{\mathcal O}, T_1, T_2, \dots)},$$
where $\overline{(\,\ast\,)}$ denotes topological closure. One can show without difficulty that the $\frak m_R$--adic topology and the original, profinite limit topology on $R$ are identical.

\begin{remark} From the point of view of the theory of profinite group algebras, one can show that if $\Gamma\simeq\Bbb Z_p^{\aleph_0}$ (a product of countably many copies of $\Bbb Z_p$), then its profinite group algebra $\mathcal O\llbracket\Gamma\rrbracket$ is isomorphic to $R$, as  topological $\mathcal O$--algebras. (See Proposition 2.9 in \cite{bleypop}.)
\end{remark}

 The technical difficulty in this case is that the ring $R$ is no longer Noetherian and, as a consequence, its ideals may not be closed. For example, its maximal ideal is not finitely generated. Also, finitely generated $R$--modules may not be finitely presented, hence, although compact, they might not be Hausdorff. Note that in the arguments of the previous sections, the Noetherian property guarantees that all the emerging submodules are automatically finitely generated
and hence compact and Hausdorff, so we may take projective limits without losing exactness of sequences.
When Noetherianess fails, we need to pay extra attention to whether the modules involved are compact
and Hausdorff or not. Thus, in this section, we first work with a special type of projective systems, which we call ``well-structured''. In the next section, we make use of the results in this section to prove the most general result, which is the ultimate goal of this paper.

\subsection{The general infinite rank set--up} \label{sectionwellstructure}
We start by introducing some additional notation.
As before, we let $\omega_m(T)=(T+1)^{p^m}-1$ be the $m$--th Weierstrass polynomial. We consider the following closed $R$--ideals:
	$$I_n\coloneqq\overline{\langle T_{n+1},T_{n+2},\cdots\rangle},\quad \Omega_{n,m}=\langle\omega_m(T_1),\omega_m(T_2),\cdots,\omega_m(T_n)\rangle,\quad J_{n,m}=I_n+\Omega_{n,m}.$$
It is very clear that $R/I_n\cong R_n$, canonically, as topological $\mathcal O$--algebras. If we let
	$$R_{n,m}:=R/J_{n,m}\cong R_n/\Omega_{n,m}R_n,$$ 
then we have canonical isomorphisms of topological $\mathcal O$--algebras   
	 $$R_n\cong\varprojlim_m R_{n,m}, \qquad R\cong\varprojlim\limits_{n}(\varprojlim\limits_{m} R_{n,m})\cong \varprojlim\limits_{m}(\varprojlim\limits_{n} R_{n,m})\cong \varprojlim\limits_{m,n}R_{n,m}.$$
The last two isomorphisms above are a consequence of the results in Appendix A.	

\begin{lemma}\label{ideals-in-R}
Let $J$ be an ideal in $R$. Then $(J+I_n)$ is closed, for all $n$, and 
$$\bigcap_n(J+I_n)=\overline J.$$
\end{lemma}
\begin{proof}
 Let $\pi_n:R\to R_n$ denote the natural surjective, continuous morphism of compact $\mathcal O$--algebras. Since $I_n=\ker(\pi_n)$, we have an equality of $R$--ideals
 $$(J+I_n)=\pi_n^{-1}(\pi_n(J)).$$
 Since $R_n$ is Noetherian and compact, its ideal $\pi_n(J)$ is finitely generated and compact, therefore closed. Since $\pi_n$ is continuous, the above equality shows that $(J+I_n)$ is closed. 
Now, the equality in the statement follows directly from Lemma \ref{toptrick}.
\end{proof}
 \medskip
 
 In this context, a projective system of $R$--modules consists of the following data: For all $(n,m)$, we let $A_{n,m}$ be an $R_{n,m}$--module, such that $(A_{n,m})_{n,m}$ forms a projective system in the category of $R$-modules, in the sense of Appendix A. For this system, we let
$$X_n\coloneqq\varprojlim\limits_{m} A_{n,m}. \qquad X\coloneqq\varprojlim\limits_{n}X_n=\varprojlim\limits_{n}(\varprojlim\limits_{m} A_{n,m})=\varprojlim\limits_{n,m}A_{n,m}.$$
Note that $(X_n)_n$ itself is a projective system of $R_n$--modules, with projective limit $X$. 
\medskip

\subsection{Well--structured projective systems.}\label{well-structured-section} Next, we define what we call a ``well-structured'' projective system in the context described above, a condition much stronger than surjectivity.
\begin{definition}\label{well-structured-def}
	With notations as above, we say that the projective system $(A_{n,m})_{n,m}$ is ``well-structured'' if,  for all $n,m, n', m'$ with $n'\geqslant n$ and $m'\geqslant m$, the transition map $A_{n', m'}\to A_{n, m}$ satisfies the following equivalent properties.
	\begin{enumerate}
    \item Induces an isomorphism of $R$--modules
$$A_{n',m'}\otimes_R(R/J_{n,m})\cong A_{n,m}.$$ 
\item Induces an isomorphism of $R_{n'}$--modules
$$A_{n',m'}\otimes_{R_{n'}}(R_{n'}/J_{n,m}R_{n'})\cong A_{n,m}.$$
\end{enumerate}
\end{definition}
\medskip

Our goal in this section is to prove the following result.
\begin{theorem}\label{theorem2}
	With the above notations, assume that the projective system $(A_{n,m})_{n,m}$ satisfies the following two properties:
	\begin{enumerate}
	\item[(i)] It is well-structured.
	\item[(ii)] Its projective limit $X\coloneqq\varprojlim\limits_{n,m}A_{n,m}$ is a finitely generated $R$--module.
	\end{enumerate}
	If $\iota$ denotes the natural identification $\iota: R\xrightarrow{\sim}\varprojlim\limits_{n,m}R_{n,m}$, then
	$$\iota\big(\,\overline{\fitzero^0_R(X)}\,\big)=\varprojlim\limits_{n,m}\fitzero^0_{R_{n,m}}(A_{n,m}).$$
    Further, if $X$ is a finitely presented $R$--module, then we have
    $$\iota\big(\,{\fitzero^0_R(X)}\,\big)=\varprojlim\limits_{n,m}\fitzero^0_{R_{n,m}}(A_{n,m}).$$
\end{theorem}

\begin{remark}\label{finite-pres-remark}
Note that if $X$ is finitely presented, then its $R$--Fitting ideal is finitely generated and therefore closed (because $R$ is compact.) Therefore, the second equality in the statement of the theorem is a direct consequence of the first. 
\end{remark}

As before, we record the following obvious corollary.

\begin{corollary}\label{corollary-theorem2}
The statement of Theorem \ref{theorem2} holds true if one replaces the ring of coefficients $\mathcal O$ with an arbitrary (not necessarily local) $p$--admissible ring.
\end{corollary}

\medskip

First, we need to prove a couple of technical Lemmas. Before we begin, note that since $X$ is a finitely generated $R$--module and the maps $X\to X_n$ and $X\to A_{n, m}$ are surjective, the $R_n$--modules $X_n$ and $A_{n, m}$ are finitely generated, Hausdorff and compact, for all $(n, m)$.
\begin{lemma}\label{lemma4.1}
	The natural projection map $X_n\to A_{n,m}$ gives a natural isomorphism $$X_n/J_{n,m}X_n\cong A_{n,m}$$
\end{lemma}
\begin{proof}
Fix $(n, m)$. By the well--structured property, the maps $A_{n,m'}\to A_{n,m}$ induce isomorphisms of compact, Hausdorff $R_n$--modules
$$A_{n, m'}\otimes_{R_n}R_n/J_{n,m}R_n\cong A_{n,m}, \qquad \text{ for all }m'\geqslant m,$$
which are compatible with the transition maps. Now, Corollary \ref{em3} applied in the context of finitely generated (therefore finitely presented, compact, Hausdorff) $R_n$--modules, followed by the above isomorphisms, leads to the desired isomorphism of $R_n$--modules 
\begin{align*}
X_n/J_{n, m}X_n \cong X_n\otimes_{R_n}R_n/J_{n,m}R_n
&\cong (\varprojlim_{m'}A_{n, m'})\otimes_{R_n}R_n/J_{n,m}R_n\\
&\cong \varprojlim_{m'}(A_{n, m'}\otimes_{R_n}R_n/J_{n,m}R_n)
\cong \varprojlim_{m'}A_{n,m}=A_{n,m}.
\end{align*}
This concludes the proof of the Lemma. \end{proof}

\begin{lemma}\label{lemma4.2}
	For all $k\geq 1$ and all $n$, there are natural isomorphisms: $$\varprojlim\limits_{m}\,(J_{n,m}A_{n+k,m})\cong I_nX_{n+k}.$$
\end{lemma}
\begin{proof} Let $s$ be the cardinality of a set of generators of the $R$--module $X$. Since the natural maps  $X\to A_{i,j}$ and $X\to X_i$ are surjective, the $R$--modules
$A_{i,j}$ and $X_i$ can also be generated by $s$ elements, for all $(i,j)$. This observation allows us to construct the following $m$--indexed projective system of exact sequences of compact, Hausdorff $R$--modules
$$0\to K_m \to J_{n, m}^s\times A_{n+k, m}^s\overset{\pi_m}{\longrightarrow} (J_{n, m}A_{n+k, m})\to 0.$$
$K_m$ is the kernel of the map $\pi_m$, which is defined by
$$\pi_m\big((x_1 ..., x_s), (a_1, ..., a_s)\big):=\sum_{i=1}^s x_i\cdot a_i.$$
Lemma \ref{compactinverse} implies that if we take the projective limit of the above sequences with respect to $m$, we obtain an exact sequence of $R$--modules 

$$0\to \varprojlim_m K_m \to (\varprojlim_m J_{n, m})^s\times X_{n+k}^s\overset{\pi}{\longrightarrow} \varprojlim_m\, (J_{n, m}A_{n+1, m})\to 0.$$
However, since $I_n$ and $\Omega_{n, m}$ are closed and $\bigcap_m\Omega_{n,m}=(0)$, Lemma \ref{toptrick} gives an equality
$$\varprojlim_m\, J_{n, m}=\bigcap_m J_{n, m}=\bigcap_m (I_n+\Omega_{n,m})=I_n.$$
Therefore, the map $\pi:=\varprojlim_m\pi_m$  can be rewritten as
$$\pi: I_n^s\times X_{n+k}^s\to \varprojlim_m\,( J_{n, m}A_{n+k, m})\subseteq X_{n+k},\qquad \pi_m\big((x_1, ..., x_s), (a_1, ..., a_s)\big):=\sum_{i=1}^s x_i\cdot a_i.$$
This shows ${\rm Im}(\pi)=I_nX_{n+k}$, which concludes the proof of the Lemma.
\end{proof}

\begin{lemma}\label{lemma4.3}
	For all $n, k\geq 1$, the natural map $X_{n+k}\to X_n$ induces an isomorphism $$X_{n+k}/I_nX_{n+k}\cong X_n.$$
\end{lemma}
\begin{proof}
	We consider the following projective system of compact, Hausdorff $R_{n+k}$--modules.
	$$0\to J_{n,m}A_{n+k,m}\to A_{n+k,m}\to A_{n,m}\to0.$$
	Take a projective limit with respect to $m$ and apply the previous Lemma to obtain an exact sequence of $R_{n+k}$--modules
	\begin{gather}
		0\to I_nX_{n+k}\to X_{n+k}\to X_n\to 0.
	\end{gather}	
	This settles the proof of the Lemma.
\end{proof}

\begin{lemma}\label{lemma4.4}
	For each $n$, the natural map $X\to X_n$ induces an isomorphism of $R$--modules
    $$ X/I_nX\cong X_n.$$
\end{lemma}
\begin{proof}
	By the remark above, for any
	$n'\geqslant n$, there is a short exact sequence: $$0\to I_nX_{n'}\to X_{n'}\to X_n\to0.$$
	Note that for each $n'$, $X_{n'}$ and $I_nX_{x'}$ are finitely generated $R_{n'}$-modules, hence they are compact and Hausdorff as topological groups. As a result, Lemma \ref{compactinverse} can still be applied here. Thus the projective limit with respect to $n'$ is exact.
	\begin{gather}\label{thm2-step4-eq1}
		0\to\varprojlim\limits_{n'}(I_nX_{n'})\to X\to X_n\to0
	\end{gather}
	Thus in order to prove our claim, we need to show that there is a natural isomorphism 
    $$\varprojlim\limits_{n'}(I_nX_{n'})\cong I_nX.$$ 
    The proof of this is identical to the proof of Lemma \ref{lemma4.2}.  \end{proof}
    
\medskip

\begin{proof}[Proof of Theorem \ref{theorem2}] Now, we are in the position to finish the proof of the main result in this section. To simplify notations, we set $$\mathscr{F}\coloneqq\fitzero_{R}^0(X),\qquad F_n\coloneqq\fitzero_{R_n}^0(X_n),\qquad F_{n,m}\coloneqq\fitzero_{R_{n,m}}^0(A_{n,m}).$$
Then, the statement of the theorem is equivalent to the natural identification of $R$-ideals:
	$$\overline{\mathscr{F}}\cong\varprojlim\limits_{n,m}F_{n,m}.$$
	Noting that $F_n=\fitzero_{R_n}^0(X_n)$ and $X_n\cong X/I_nX$, we have:
	\begin{align*}
		F_n=&\fitzero_{R_n}^0(X/I_nX)=
			   	\fitzero_{R_n}^0(X\otimes_R R_n)\\
			   	=&\fitzero_{R}^0(X)_{R_n}
			   	=\big(\mathscr F+I_n\big)/I_n.
	\end{align*}
	  Note that by Lemma \ref{ideals-in-R}, the ideal $(\mathscr{F}+I_n)$ is closed and therefore compact, for all $n$.
	  We consider the projective system of exact sequences of compact, Hausdorff $R_n$--modules:
	  $$0\to I_n\to\big(\mathscr{F}+I_n\big)\to F_n\to0.$$
	 When taking the projective limit, we obtain an exact sequence
	  $$0\to \bigcap\limits_{n}I_n\to\bigcap\limits_{n}\big(\mathscr{F}+I_n\big)\to \varprojlim\limits_{n}F_n\to0.$$
	  Since $R\cong\varprojlim_nR_n$, we have $\bigcap\limits_{n}I_n=0.$ Therefore, when applying  Lemma \ref{ideals-in-R}, we get $$\overline{\mathscr{F}}=\bigcap\limits_{n}\big(\mathscr{F}+I_n\big)\cong \varprojlim\limits_{n}F_n.$$  However, by Theorem \ref{theorem1} applied to the data $R_n$ and $X_n:=\varprojlim\limits_m A_{n, m}$, we have a natural isomorphism of $R_n$--ideals: $$F_n\cong\varprojlim\limits_{m}F_{n,m}.$$ Therefore, we obtain the desired natural isomorphism of $R$--ideals $$\overline{\mathscr{F}}\cong \varprojlim_{n,m}F_{n,m},$$
      which settles the proof of the theorem.
\end{proof}

\section{The General ``Infinite-Rank'' Case}\label{general-section}
In this section, we will replace the ``well-structured'' condition on the projective system $(A_{n,m})_{n,m}$ with the weaker, surjectivity condition. Therefore, with notations as in the previous section, our main goal is the proof of the following. 
\begin{theorem}\label{maintheorem}
	For data as in Section \ref{sectionwellstructure}, assume that the projective system $(A_{n,m})_{n,m}$ satisfies the following two properties:
	\begin{itemize}
	\item[(i)] It is surjective .
	\item[(ii)] Its limit $X$ is a finitely generated $R$--module.
	\end{itemize}
	If $\iota$ denotes the natural identification $\iota: R\xrightarrow{\sim}\varprojlim\limits_{n,m}R_{n,m}$, then
	$$\iota(\overline{\fitzero^0_R(X)})=\varprojlim\limits_{n,m}\fitzero^0_{R_{n,m}}(A_{n,m}).$$
Further, if $X$ is a finitely presented $R$--module, then we have
    $$\iota\big(\,{\fitzero^0_R(X)}\,\big)=\varprojlim\limits_{n,m}\fitzero^0_{R_{n,m}}(A_{n,m}).$$
    
\end{theorem}

Before we begin the proof, we record, as usual, the obvious corollary.

\begin{corollary}\label{corollary-maintheorem}
The statement of Theorem \ref{maintheorem} holds true if one replaces the ring of coefficients $\mathcal O$ with an arbitrary (not necessarily local) $p$--admissible ring.
\end{corollary}

The proof of the above theorem requires a few technical lemmas. First, we set up additional notations and definitions.
First, we define the $R_{k,l}$--modules: 
$$A_{n,m}^{(k,l)}\coloneqq A_{n,m}/J_{k,l}A_{n,m}.$$
Note that, if $k\geqslant n$ and $l\geqslant m$, we have $J_{k,l}A_{n,m}=0$ and therefore $A_{n,m}^{(k,l)}=A_{n,m}$. Thus,
$$\varprojlim\limits_{k,l}A_{n,m}^{(k,l)}=A_{n,m}, \qquad \varprojlim\limits_{n,m}\varprojlim\limits_{k,l}A_{n,m}^{(k,l)}=\varprojlim\limits_{n,m}A_{n,m}=X,$$
where the projective limits are taken with respect to the transition maps induced by those in the original projective system $\{A_{n,m}\}$, keeping in mind that $J_{k,l}\subseteq J_{k', l'}$, whenever $k\geq k'$ and $l\geq l'$. 
Next, we define the $R_{k,l}$--module
$$X_{k,l}\coloneqq\varprojlim\limits_{n,m}A_{n,m}^{(k,l)}.$$
Next, we note that we have equalities and natural isomorphisms
\begin{align*}
	\varprojlim\limits_{k,l}X_{k,l}=\varprojlim\limits_{k,l}\varprojlim\limits_{n,m}A_{n,m}^{(k,l)}
	\cong \varprojlim\limits_{n,m}\varprojlim\limits_{k,l}A_{n,m}^{(k,l)}
	=\varprojlim\limits_{n,m}A_{n,m}
	=X.
\end{align*}
The second isomorphism above is a mild generalization of Lemma \ref{commutativity} (see the Remark at the end of Appendix A), the proof of which we leave to the interested reader.

\begin{lemma}\label{lemma5.1}
	The projective system $(X_{k,l})_{k,l}$ is ``well-structured'' in the sense of Section \ref{well-structured-section}.
\end{lemma}
\begin{proof} We will use Definition \ref{well-structured-def}(2) to prove that $(X_{k,l})_{k,l}$ is well--structured. For that, 
fix $(k', l')$ and $(k, l)$, with $k'\geqslant k$ and $l'\geqslant l$. Since $J_{k',l'}\subseteq J_{k,l}$, an easy commutative algebra exercise shows that we have a canonical  isomorphism of $R$--algebras
$$(R/J_{k',l'})\otimes_{R_{k'}}(R_{k'}/{J_{k,l}R_{k'}})\simeq R/J_{k,l}.$$
Now, combine this isomorphism with Corollary \ref{em3}, applied to the category of finitely generated (therefore finitely presented, Hausdorff, compact) $R_{k'}$--modules, to conclude that we have canonical isomorphisms
\begin{align*}
X_{k',l'}\otimes_{R_{k'}}(R_{k'}/J_{k,l}R_{k'})&=\Big(\varprojlim\limits_{n,m}(A_{n,m}\otimes_R(R/J_{k',l'}))\Big)\otimes_{R_{k'}}(R_{k'}/J_{k,l}R_{k'})\\
&\cong \varprojlim\limits_{n,m}\Big(A_{n,m}\otimes_R(R/J_{k',l'}\otimes_{R_{k'}}R_{k'}/J_{k,l}R_{k'})\Big)\\
&\cong \varprojlim\limits_{n,m}(A_{n,m}\otimes_RR/J_{k,l})=X_{k,l}.
\end{align*}
This settles the proof of the lemma.
\end{proof}
Let $\mathscr F:={\fitzero_{R}^0(X)}$. Since the projective system $(X_{k,l})_{k,l}$ is ``well-structured'', Theorem \ref{theorem2} gives a canonical isomorphism:
\begin{gather}\label{after_lemma5.1}
	\overline{\mathscr{F}}\cong\varprojlim\limits_{k,l}\fitzero_{R_{k,l}}^0(X_{k,l}).
\end{gather}
Next, define $F_{n,m}\coloneqq\fitzero_{R_{n,m}}^0(A_{n,m})$, and $\mathscr{F}_{n,m}\coloneqq\fitzero_{R}^0(A_{n,m})$ and note that we have
$$F_{n,m}\cong\Big(\mathscr{F}_{n,m}+J_{n,m}\Big)/J_{n,m}.$$
Note that the $R$--ideal $\mathscr{F}_{n,m}+J_{n,m}=(\mathscr{F}_{n, m}+I_n)+\Omega_{n,m}$ is compact, by Lemma \ref{ideals-in-R}. Therefore, we get a projective system of short exact sequences of compact Hausdorff $R$--modules:
$$0\to J_{n,m}\to \mathscr{F}_{n,m}+J_{n,m}\to F_{n,m}\to0,$$
whose projective limit stays exact, by Lemma \ref{compactinverse}:
$$0\to\varprojlim\limits_{n,m}J_{n,m}\to\varprojlim\limits_{n,m}\Big(\mathscr{F}_{n,m}+J_{n,m}\Big)\to\varprojlim\limits_{n,m}F_{n,m}\to0.$$
Now, note that we have equalities 
$$\varprojlim\limits_{n,m}J_{n,m}=\bigcap\limits_{n,m}J_{n,m}=\langle 0\rangle, \qquad \varprojlim\limits_{n,m}\Big(\mathscr{F}_{n,m}+J_{n,m}\Big)=\bigcap\limits_{n,m}\Big(\mathscr{F}_{n,m}+J_{n,m}\Big).$$  
When combined with the exact sequence above, these equalities give a natural isomorphism 
\begin{equation}\label{before_lemma5.2}\varprojlim\limits_{n,m}F_{n,m}\cong\bigcap\limits_{n,m}\Big(\mathscr{F}_{n,m}+J_{n,m}\Big).
\end{equation}

\begin{lemma}\label{lemma5.2}
	We have $$\bigcap\limits_{n,m}\Big(\mathscr{F}_{n,m}+J_{n,m}\Big)=\bigcap\limits_{n,m}\Big(\overline{\mathscr{F}_{n,m}}\Big),$$
	where $\overline{\mathscr{F}_{n,m}}$ denotes the closure of the ideal $\mathscr{F}_{n,m}$ in $R$.
\end{lemma}
\begin{proof}
	First of all, since $\mathscr{F}_{n,m}+J_{n,m}$ is closed, we have $$\overline{\mathscr{F}_{n,m}}\subseteq\mathscr{F}_{n,m}+J_{n,m},$$ thus
	$$\bigcap\limits_{n,m}\Big(\overline{\mathscr{F}_{n,m}}\Big)\subseteq\bigcap\limits_{n,m}\Big(\mathscr{F}_{n,m}+J_{n,m}\Big).$$
	To prove the other inclusion, we take $\xi\in\bigcap\limits_{n,m}\Big(\mathscr{F}_{n,m}+J_{n,m}\Big).$ Then for each pair $(n,m)$ there exist $f_{n,m}\in\mathscr{F}_{n,m}$ and $j_{n,m}\in J_{n,m}$ such that
	$$\xi=f_{n,m}+j_{n,m}.$$
	Note that $\mathbb{N}\times\mathbb{N}$ is a directed partially ordered set in the sense of the appendix (see the remark therein), and hence we may regard $\{f_{n,m}\}_{n,m}$ and $\{j_{n,m}\}_{n,m}$ as \textit{nets} $$f_{\bullet}:\mathbb{N}\times\mathbb{N}\to R,\qquad  j_{\bullet}:\mathbb{N}\times\mathbb{N}\to R.$$ 
	
	Since $R$ is compact, Proposition \ref{uniquenetlimit} shows that there exists a co-final subset $T\subseteq\mathbb{N}\times\mathbb{N}$ such that the subnets $f|_{T}:T\to R$ and $j|_{T}: T\to R$ are convergent, and since $R$ is Hausdorff, the  limit point of any net is unique. Now, for each $(n,m)$, there exists $s=(n_s,m_s)\in T$ such that $n_s\geqslant n, m_s\geqslant m$, for $T$ is co-final. Therefore, we have $$f_{t}\in\mathscr{F}_{t}\subseteq\mathscr{F}_{n,m},$$ whenever $t\geqslant s$. Hence, we have $$\mathfrak{f}\coloneqq\lim\limits_{t}f_t=\lim\limits_{t\geqslant s}f_t\in\overline{\mathscr{F}_{n,m}}.$$
	Since this is true for any pair $(n,m)$, we must have $\mathfrak{f}\in\bigcap\limits_{n,m}\Big(\overline{\mathscr{F}_{n,m}}\Big)$.
	The very same argument applies to the net $j_{\bullet}$, and we get
	$$\mathfrak{j}\coloneqq\lim\limits_{t}j_t\in\bigcap\limits_{n,m}\Big(\overline{J_{n,m}}\Big)=\bigcap\limits_{n,m}J_{n,m}=\{0\}.$$ Hence, $\mathfrak{j}=0$.
	Now, since $+:R\times R\to R$ is continuous, we have
	$$\xi=\lim\limits_{t}\xi=\lim\limits_{t}(f_{t}+j_t)=(\lim\limits_{t}f_t)+(\lim\limits_{t}j_t)=\mathfrak{f}+\mathfrak{j}=\mathfrak{f}+0\in\overline{\mathscr{F}_{n,m}}.$$
	This settles the proof of Lemma \ref{lemma5.2}.
\end{proof}
\begin{corollary}We have a natural isomorphism of $R$--ideals $$\label{corollary-section5} \varprojlim\limits_{n,m}F_{n,m}\cong\bigcap\limits_{n,m}\Big(\overline{\mathscr{F}_{n,m}}\Big),$$

\end{corollary}
\begin{proof}
Combine Lemma \ref{lemma5.2} with equality \eqref{before_lemma5.2}.
\end{proof}

Next, we give a generalization of Lemma \ref{lemma2.2}.
\begin{lemma}\label{lemma-gen-finite-rank}
	Let $A$ be a PID and let $(S,\mathfrak{m})$ be a local $A$-algebra, which is compact and Hausdorff in its $\frak m$--adic topology. Assume further that $S$ is a free $A$-module of finite rank.
	Let $(B_{n,m})_{n,m}$ be a projective system of $S$-modules in the sense of the Appendix, and let $B\coloneqq\varprojlim\limits_{n,m} B_{n,m}$ be the projective limit.
	Assume that \begin{itemize}
	
	\item[(i)] The projective system of modules $(B_{n,m})_{n,m}$ is surjective.
	
	\item[(ii)] The projective limit $B$ is a finitely generated, torsion module over $S$.
	\end{itemize}
	Then, there is an equality of ideals of $S$:
	$$\fitzero_S^0(B)=\varprojlim\limits_{n,m}\fitzero_S^0(B_{n,m})=\bigcap\limits_{n,m}\fitzero_S^0(B_{n,m}).$$
\end{lemma}
\begin{proof}
	Let $\widetilde{B}_{n}\coloneqq\varprojlim\limits_{m}B_{n,m}$.
	Then, we may apply Lemma \ref{lemma2.2} twice to obtain:
	\begin{align*}
		\fitzero_S^0(B)=\fitzero_S^0(\varprojlim\limits_{n}\widetilde{B}_{n}) =&\varprojlim\limits_{n}\fitzero_S^0(\widetilde{B}_{n})\\
									  =&\varprojlim\limits_{n}\fitzero_S^0(\varprojlim\limits_{m}B_{n,m})
									  =\varprojlim\limits_{n}\varprojlim\limits_{m}\fitzero_S^0(B_{n,m}).
	\end{align*}
	The second equality in the statement follows from the surjectivity of the system which forces inclusions at the level of Fitting ideals.
\end{proof}
\begin{proof}[Proof of Theorem \ref{maintheorem}] We are ready to prove the main result of this section. First, fix a pair $(k,l)$ and consider the following data:
 $$S:=R_{k,l},\qquad
 B_{n,m}=A_{n,m}\otimes_R R_{k,l},\qquad
B=X_{k,l}=\varprojlim\limits_{n,m}B_{n,m}.$$
Since $R_{k,l}\cong R_k/\Omega_{k,l}$, $S$ is a free $\mathcal O$--module of rank $p^{kl}$. As $\mathcal O$ is $p$--adimissible, $\mathcal O$ is a free module of finite rank over a PID $A$. Consequently, $S$ is a free $A$--module of finite rank. It is obvious that $B$ is a finitely generated $S$--module. Hence, Lemma \ref{lemma-gen-finite-rank} gives:
\begin{align*}
	\fitzero_{R_{k,l}}^0(X_{k,l})=\fitzero_{R_{k,l}}^0\Big(\varprojlim\limits_{n,m}\big(A_{n,m}\otimes_R R_{k,l}\big)\Big)
	&=\varprojlim\limits_{n,m}\fitzero_{R_{k,l}}^0\big(A_{n,m}\otimes_R R_{k,l}\big)\\
	&=\varprojlim\limits_{n,m}\fitzero_{R}^0\big(A_{n,m}\big)_{R_{k,l}}
	\cong \varprojlim\limits_{n,m}\Big(\big(\mathscr{F}_{n,m}+J_{k,l}\big)/J_{k,l}\Big).
\end{align*}
Combining  the above with \eqref{after_lemma5.1} leads to the following.
\begin{align*}
	\mathscr{F}\cong \varprojlim\limits_{k,l}\fitzero_{R_{k,l}}^0(X_{k,l})
						   \cong\varprojlim\limits_{k,l}\varprojlim\limits_{n,m}\Big(\big(\mathscr{F}_{n,m}+J_{k,l}\big)/J_{k,l}\Big)
						   \cong\varprojlim\limits_{n,m}\varprojlim\limits_{k,l}\Big(\big(\mathscr{F}_{n,m}+J_{k,l}\big)/J_{k,l}\Big).
\end{align*}
An argument similar to that used in the proof of Lemma \ref{lemma5.2} gives
$$\varprojlim\limits_{k,l}\Big(\big(\mathscr{F}_{n,m}+J_{k,l}\big)/J_{k,l}\Big)\cong\overline{\mathscr{F}_{n,m}}.$$
Therefore, by Corollary \ref{corollary-section5}, we obtain
\begin{align*}
	\mathscr{F}\cong&\varprojlim\limits_{n,m}\varprojlim\limits_{k,l}\Big(\big(\mathscr{F}_{n,m}+J_{k,l}\big)/J_{k,l}\Big)
	\cong\varprojlim\limits_{n,m}\Big(\overline{\mathscr{F}_{n,m}}\Big)
	=\bigcap\limits_{n,m}\Big(\overline{\mathscr{F}_{n,m}}\Big)
	=\varprojlim\limits_{n,m}F_{n,m},
\end{align*}
This settles our proof of Theorem \ref{maintheorem}.
\end{proof}

\begin{appendices}
\section{Commutativity of Double Inverse Limits}
	In this appendix, we prove the following lemma, which is frequently used in the arguments contained in the previous sections of this paper.
	\begin{lemma}[Commutativity of Inverse Limits]\label{commutativity}
	Let $\mathcal{C}$ be a complete category, and let $\{C_{m,n}\}$ be a family of objects in $\mathcal{C}$ indexed by $(m,n)\in\mathbb{N}\times\mathbb{N}$. Assume that there are two families of morphisms $\{v_{m,n}:C_{m+1,n}\to C_{m,n}\}_{(m,n)}$ and $\{h_{m,n}:C_{m,n+1}\to C_{m,n}\}_{(m,n)}$, such that the following equality holds for any pair $(m,n)$:
	$$v_{m,n}\circ h_{m+1,n}=h_{m,n}\circ v_{m,n+1}.$$
	Then, in $\mathcal C$ there exist canonical objects and a canonical isomorphism between them
	$$\varprojlim\limits_{m}\,(\varprojlim\limits_{n}C_{m,n})\cong\varprojlim\limits_{n}\,(\varprojlim\limits_{m}C_{m,n}).$$
	In particular, one can denote by $\varprojlim\limits_{m,n}C_{m,n}$ without introducing any ambiguity any of the double projective limits above.
\end{lemma}
\begin{proof}
	First, let us make sense of the double projective limits above. For each $m, n$, let 
    $$X_m\coloneqq\varprojlim\limits_{n}C_{m,n}, \qquad p_{m,n}: X_m\to C_{m,n}$$ 
    where the projective limit is taken with respect to the transition morphisms $h_{m,n}$ (since we live in a complete category, the limit exists) and $p_{m,n}$ is the ensuing natural projection. By the universality property of projective limits, we must have
	\begin{gather}\label{eq0}
		p_{m,n}=h_{m,n}\circ p_{m,n+1}.
	\end{gather}
	Now, all $m, n$, when chasing the commutative diagram below, it is easy to see that $$v_{m,n}\circ p_{m+1,n}=h_{m,n}\circ v_{m,n+1}\circ p_{m+1,n+1}.$$
	Please ignore the dotted arrow in the diagram below, for the moment!
	\[\begin{tikzcd}
		X_{m+1}  \ar[dd,dashrightarrow, "x_m"]\ar[dr,"p_{m+1,n+1}"] \ar[drr, "p_{m+1,n}", bend left=20]\\		
		&C_{m+1,n+1}\arrow[d,"v_{m,n+1}"]\arrow[r,"h_{m,n+1}"]& C_{m+1,n} \arrow[d,"v_{m,n}"]\\
		X_m\arrow[r,"p_{m,n+1}"'] & C_{m,n+1}\arrow[r,"h_{m,n}"'] &C_{m,n}
	\end{tikzcd}\]
 	The universal property of projective limits gives a unique morphism $x_m: X_{m+1}\to X_{m}$ (see the dotted arrow in the diagram above), such that
 	  \begin{gather}\label{eq1}
 		 p_{m,n}\circ x_m=v_{m,n}\circ p_{m+1,n}.
 		\end{gather}
 	Regarding the maps $x_m:X_{m+1}\to X_m$ as transition maps, we can consider the following projective limit and ensuing projections:
    $$X\coloneqq\varprojlim\limits_{m}X_m, \qquad P_m:X\to X_m, \qquad \text{with } P_m=x_m\circ P_{m+1}.$$ 
    
    Similarly, with respect to the maps $v_{m,n}$ we have a projective limit an ensuing projections $$Y_n\coloneqq\varprojlim\limits_{m}C_{m,n}, \qquad q_{m,n}: Y_n\to C_{m,n},$$ 
    as well as maps 
    $y_n:Y_{n+1}\to Y_n$ and a corresponding projective limit and ensuing projections
    $$Y\coloneqq\varprojlim\limits_{n}Y_n, \qquad Q_n: Y\to Y_n.$$
    As before, we have an equality of maps
 	\begin{gather}\label{eq2}
 		q_{m,n}\circ y_n=h_{m,n}\circ q_{m,n+1}.
 	\end{gather}
 	\medskip
 	
 	Now for each $n$, consider the commutative diagrams:
 	\[ 
    \begin{array}{cc}
 	\begin{tikzcd}
 		X\arrow[rr,"P_{m+1}"]\arrow[drr,"P_m"']& &X_{m+1}\arrow[d,"x_m"]\arrow[r,"p_{m+1,n}"]&C_{m+1,n}\arrow[d,"v_{m,n}"]\\
 		& &X_m\arrow[r,"p_{m,n}"']&C_{m,n}
 	\end{tikzcd}
    &\qquad\qquad 
    \begin{tikzcd}
 		& &C_{m+1,n}\ar[dd,"v_{m,n}"]\\
 		X\arrow[urr,"p_{m+1,n}\circ P_{m+1}" pos=0.8,bend left=15]\ar[drr,"p_{m,n}\circ P_{m}"'pos=0.6, bend right=15]\\
 		& &C_{m,n} 		
 	\end{tikzcd}
    \end{array}
\]
In the left diagram: the right square is commutative because of Formula (\ref{eq1}), and the left triangle is commutative by the definition of $X$. Hence, the right diagram is commutative as well.  Now, the right diagram above, combined with the universality property of the projective limit $Y_n$, leads to the existence of a unique morphism 
 	\begin{gather}\label{eq3}
 		\psi_n: X\to \varprojlim\limits_{m}C_{m,n}=Y_n, \qquad\text{such that }\quad   p_{m,n}\circ P_m=q_{m,n}\circ\psi_n,
 	\end{gather}
which is illustrated by the commutative diagram below:
	 	\[\begin{tikzcd}
		& & &C_{m+1,n}\ar[dd,"v_{m,n}"]\\
		X\arrow[urrr,"p_{m+1,n}\circ P_{m+1}" pos=0.8, bend left=20]\ar[drrr,"p_{m,n}\circ P_{m}"'pos=0.6,  bend right=20]\arrow[rr, ,dashrightarrow, "\psi_n"]& & Y_n\arrow[ur, "q_{m+1,n}"]\arrow[dr,"q_{m,n}"']\\
		& & &C_{m,n} 		
	\end{tikzcd}\]
	Next, we consider the following diagram.
		\[
		\begin{tikzcd}
			X\arrow[rr,"\psi_{n+1}"]\arrow[drr,"\psi_n"']& &Y_{n+1}\arrow[d,"y_n"]\arrow[r,"q_{m,n+1}"]&C_{m,n+1}\arrow[d,"h_{m,n}"]\\
			& &Y_n\arrow[r,"q_{m,n}"']&C_{m,n}
		\end{tikzcd}
		\]
		We would like to show that the left triangle above is commutative. By the universality property of the projective limit $Y_n=\varprojlim\limits_{m}C_{m,n}$, it suffices to show that 
        $$q_{m,n}\circ y_n\circ \psi_{n+1}=q_{m,n}\circ \psi_n, \qquad\text{ for all }m.$$
		This is indeed true, as shown by the following equalities.
		\begin{align*}
			(q_{m,n}\circ y_n)\circ \psi_{n+1}&\xlongequal{(\ref{eq2})} (h_{m,n}\circ q_{m,n+1})\circ \psi_{n+1} \xlongequal{\phantom{(\ref{eq2})}} h_{m,n}\circ (q_{m,n+1}\circ \psi_{n+1})\\
			&\xlongequal{(\ref{eq3})} h_{m,n}\circ (p_{m,n+1}\circ P_m)
			\xlongequal{\phantom{(\ref{eq3})}} (h_{m,n}\circ p_{m,n+1})\circ P_m\\
			&\xlongequal{(\ref{eq0})}  p_{m,n}\circ P_m\xlongequal{(\ref{eq3})} q_{m,n}\circ \psi_n.
		\end{align*}
		Consequently, by the universal property of $Y=\varprojlim\limits_{n}Y_n$, there exists a unique morphism 
		\begin{gather}\label{eq4}
			\psi: X\to Y, \qquad \text{ such that }  Q_n\circ \psi=\psi_n,\quad  \text{ for all }n.
		\end{gather}
		Similarly, for each $m$, there exists a unique morphism $\eta_{m}:Y\to X_m$, satisfying
		\begin{gather}\label{eq5}
			q_{m,n}\circ Q_{n}=p_{m,n}\circ \eta_{m},\quad \text{ for all }m.
		\end{gather}
		This is illustrated by the commutative diagram below:
		\[\begin{tikzcd}
			& &Y\ar[ddddl,"q_{m,n+1}\circ Q_{n+1}"', bend right]\ar[ddddr, "q_{m,n}\circ Q_n", bend left]\arrow[dd,"\eta_{m}"]\\
			& & &\\
			& &X_m\arrow[ddl,"p_{m,n+1}"',pos=0.2]\arrow[ddr,"p_{m,n}"]\\
			& & &\\
			&C_{m,n+1}\arrow[rr,"h_{m,n}"'] & &C_{m,n}
		\end{tikzcd}\]
		Hence, by the universal property of $X=\varprojlim\limits_{m}X_m$, there exists a unique morphism $\eta:Y\to X$, such that the following holds or all $m$:
		\begin{gather}
			P_m\circ\eta=\eta_{m}, 
		\end{gather}
	
		We would like to show that the following equalities hold:
        $$\eta\circ\psi=\operatorname{id}_{X}, \qquad \psi\circ\eta=\operatorname{id}_{Y}.$$ 
        For that, note that it suffices to show that the outer triangle of the following diagram is commutative, for all $m$.
		\[\begin{tikzcd}
			&X\arrow[rr,"\psi"]\ar[ddrr,"P_m"']& &Y\arrow[rr,"\eta"]\ar[dd,"\eta_{m}"] & &X\ar[ddll,"P_m"]\\
			& & & & &\\
			& & &X_m
		\end{tikzcd}\]
		Indeed, if above diagrams are commutative, then one has $$P_m\circ(\eta\circ\psi)=P_m=P_m\circ\operatorname{id}_X,$$
for all $m$.  Via the universal property of $X=\varprojlim\limits_{m}X_m$, this implies that
$$\eta\circ\psi=\operatorname{id}_X.$$
Note that the right triangle in the above diagram is commuter due to (A.7). Thus, one has to verify the commutativity of the left triangle. As before, this would follow if
		$$p_{m,n}\circ\eta_{m}\circ\psi=p_{m,n}\circ P_m,$$
holds for every $n$. And this is indeed true due to the following equalities.
		\begin{align*}
			(p_{m,n}\circ\eta_{m})\circ\psi\xlongequal{(\ref{eq5})}(q_{m,n}\circ Q_n)\circ\psi
			&\xlongequal{\phantom{(\ref{eq3})}}q_{m,n}\circ (Q_n\circ\psi)\\
			&\xlongequal{(\ref{eq4})}q_{m,n}\circ\psi_n
			\xlongequal{(\ref{eq3})}p_{m,n}\circ P_m.
		\end{align*}
	This settles our proof.
\end{proof}
\begin{remark}
	
	

We would like to remark that the above lemma can be easily generalized to the case where the system of modules is indexed by finitely many copies of $\mathbb{N}$. We leave the details to the interested reader.
\end{remark}

\section{An application to geometric Iwasawa theory}\label{Appendix B}
In this appendix, we give an application to the geometric, non--Noetherian Iwasawa theory of $p$--adic realizations of $1$--motives with abelian coefficients, defined over a characteristic $p$ global field. More precisely, we will give a new, very short proof of Theorem 3.15 (the Equivariant Main Conjecture for Ritter--Weiss modules) in \cite{bleypop}, as a direct application of our Theorem \ref{theorem2} above. Further applications in similar, but currently much less understood Iwasawa theoretic contexts, will follow in upcoming papers.

First, we give a brief description of the number theoretic context in loc.cit., focusing only on the details which are relevant to our immediate purposes. We will use simplified notations, which are easily translatable into the language used in \S\S3.3--3.5 =and the Appendix of \cite{bleypop}. For additional number theoretic details the reader should consult loc.cit.
\medskip

The initial arithmetic data is $\mathscr D:=(k,\, v_\infty,\, A,\, \frak f,\, \frak p)$, where $k$ is a characteristic $p$ global field of exact field of constants $\Bbb F_q$, $v_\infty$ is a fixed non--trivial valuation on $k$, $A$ is the subring of $k$ consisting of all elements integral away from $v_\infty$, $\frak f$ is an ideal in $A$, and $\frak p$ is a maximal ideal in $A$, not containing $\frak f$. To this data, one associates an infinite, Galois, abelian extension $L_\infty/k$, defined as the union
$$L_\infty:=\bigcup_n L_n,$$
where $L_n$ is the $v_\infty$--split ray--class field of $k$ of conductor $\frak f\frak p^n$. Now, the profinite, abelian Galois group $G_\infty:=G(L_\infty/k)$ sits in a short exact sequence
$$0\to G(L_\infty/L_0)\to G_\infty\to G(L_0/k)\to 0,$$
where $G(L_0/k)$ is finite and there is a topological group isomorphism $G(L_\infty/L_0)\cong\Bbb Z_p^{\aleph_0}$. An elementary argument involving pro-$p$ groups combined with class--field theory shows that if $G$ is the torsion subgroup of $G_\infty$, then $G$ is finite, $\widetilde{G_\infty}:= G_\infty/G\cong \Bbb Z_p^{\aleph_0}$ and one has a direct product decomposition
\begin{equation}\label{Ginfinity}G_\infty\simeq G\times \widetilde{G_\infty}\cong G\times \Bbb Z_p^{\aleph_0}.\end{equation}
Consequently, after an ordering of the topological generators of $\widetilde{G_\infty}$, we have isomorphisms of compact, Hausdorff $\Bbb Z_p$--algebras  
\begin{equation}\label{group-rings}R_\infty:=\Bbb Z_p\llbracket G_\infty\rrbracket \cong \Bbb Z_p[G]\llbracket\Bbb Z_p^{\aleph_0}\rrbracket\cong \Bbb Z_p[G]\llbracket T_1, T_2, \dots\rrbracket.\end{equation}
Note that the ring $R_\infty$ fits into our general theory of power series rings in countably many variables, with coefficients in the (semi-local) $p$--admissible ring $\Bbb Z_p[G]$. (See Remark \ref{examples-p-admissible}.) Therefore, Theorems \ref{maintheorem}, \ref{theorem2} and their corollaries apply to appropriate projective systems of $R_\infty$--modules.  
\medskip

Now, by using earlier work of Greither and the first author \cite{GP1} and \cite{GP2} on $p$--adic realizations of Picard $1$--motives, the authors of \cite{bleypop} associated to the data $(\mathscr D, \, L_\infty/k)$  a natural projective system 
$\{\nabla(L/k)\}_L$ of finitely generated $\Bbb Z_{p}[G(L/k)]$--modules, indexed with respect to all fields $L$, such that $k\subseteq L\subseteq L_{\infty}$ and $L/k$ is finite, and whose transition maps
$$\nabla(L'/k)\to \nabla(L/k), \qquad \text{ for all }L\subseteq L'$$
are surjective and induce isomorphisms of $\Bbb Z_p[G(L/k)]$--modules 
\begin{equation}\label{nabla-codescent}
\nabla(L'/k)_{G(L'/L)}\cong \nabla(L/k), \qquad \text{ for all }L\subseteq L'.
\end{equation}
Moreover, as it is observed in \cite{bleypop}, the results in \cite{GP1} and \cite{GP2} imply that 
\begin{equation}\label{nabla-fitt}
\fitzero^0_{\Bbb Z_p[G(L/k)]}{\nabla(L/k)}=\langle \Theta_{L/k}\rangle,\qquad\text{ for all }L,\end{equation}
where $\Theta_{L/k}$  is the value at $s=0$ of a $G(L/k)$--equivariant Artin $L$--function associated to $(\mathscr D,\, L/k).$ Further, properties of Artin $L$--functions imply that if 
$$\pi_{L'/L}:\Bbb Z_p[G(L'/k)]\to\Bbb Z_p[G(L/k)], \qquad \text{ for all }L\subseteq L',$$ 
are the surjective $\Bbb Z_p$--algebra morphisms induced by Galois restriction, then
$$\pi_{L'/L}(\Theta_{L'/k})=\Theta_{L/k}, \qquad \text{ for all }L\subseteq L'. $$
Consequently, one can define the element $\Theta_{L_\infty/k}\in R_\infty$ as follows:
$$\Theta_{L_\infty/k}:=(\Theta_{L/k})_L\in\varprojlim_L\, \Bbb Z_p[G(L/k)]=\Bbb Z_p\llbracket G_\infty\rrbracket=R_\infty.$$
This is the equivariant $p$--adic $L$--function associated to the data $(\mathscr D,\, L_\infty/k)$. It satisfies he following important property. (See Proposition 3.22 in \cite{bleypop}.)
\begin{proposition}[Bley--Popescu]\label{Bley-Pop-EMC}
The element $\Theta_{L_\infty/k}$ is a non zero--divisor in $R_\infty$.
\end{proposition}
The theorem we will focus on is the following function field analogue of the Equivaraiant Main Conjecture in classical Iwasawa theory over number fields. (See Theorem 3.15 in \cite{bleypop}.)
\begin{theorem}[Bley--Popescu]\label{main-th-appB}
The following hold for the $R_\infty$--module $\nabla_{L_\infty/k}$.
\begin{enumerate}
\item It is finitely and quadratically presented, torsion, of projective dimension $1$.
\item $\fitzero^0_{R_\infty}{\nabla(L_\infty/k)}=\langle\Theta_{L_\infty/k}\rangle.$   
\end{enumerate}
\end{theorem}
Below, we will give an alternative, relatively short proof of the above theorem, by showing that it is in fact a consequence of our Theorem \ref{theorem2} and its Corollary \ref{corollary-theorem2}.
\begin{proof}[Proof of Theorem \ref{main-th-appB}] Isomorphisms \eqref{Ginfinity} and  \eqref{group-rings} and profinite Galois theory show that there are unique intermediate fields $L_{n,m}$, with $k\subseteq L_{n,m}\subseteq L_\infty$, such that $[L_{n,m}:k]<\infty$ and, if $G_{n,m}:=G(L_{n,m}/k)$, then Galois restriction induces isomorphisms of $\Bbb Z_p$--algebras
$$R_{n,m}:=R_\infty/J_{n,m}\cong \Bbb Z_p[G_{n,m}],\qquad\text{ for all }\,n, m\geq 0,$$
where the notations are those in \S4. More explicitly, after ordering the topological generators of $\widetilde{G_\infty}$, so that 
$\widetilde{G_\infty}\cong \prod_{i=1}^\infty\Bbb Z_p$, and letting $\Gamma_{n,m}:=\prod_{i=1}^n(p^m\Bbb Z_p)^n\times \prod_{i=n+1}^\infty\Bbb Z_p$, then 
$$L_{n,m}:=L_{\infty}^{\Gamma_{n,m}}, \qquad G_{n,m}\simeq G_\infty/\Gamma_{n,m}\simeq G\times(\Bbb Z/p^m\Bbb Z)^n.$$
Note that since $\Gamma_{n',m'}\subseteq \Gamma_{n,m}$, Galois theory gives inclusions
$$L_{n,m}\subseteq L_{n',m'}, \qquad \text{ for all }n'\geq n,\, m'\geq m.$$
Further, since the intersection $\cap_{n,m}\Gamma_{n,m}$ of open subgroups of $G_\infty$ is trivial, we have
$$\bigcap_{n,m}L_{n,m}=L_\infty,$$
which makes $\{L_{n,m}\}_{n,m}$ a cofinal injective subsystem of the injective system of intermediate fields $k\subseteq L\subseteq L_\infty$, with $L/k$ finite. In particular, this shows that 
$$\nabla(L_\infty/k)=\varprojlim_{n,m}\nabla(L_{n,m}/k), \qquad \Theta_{L_\infty/k}=(\Theta_{L_{n,m}/k})_{n,m}.$$
Note that in \cite{bleypop}, the authors use the cofinal injective subsystem $\{L_n\}_n$ of intermediate fields, arising naturally from class--field theory and described in detail above. Obviously, the projective limits involved do not depend on these choices.
\medskip

We will divide the proof of Theorem \ref{main-th-appB} among several lemmas, of interest in their own right. Before we begin, we will make one more remark regarding the topological ring $R_\infty$.

\begin{remark} Note that if we proceed as in \eqref{dec-groupring-eq}, and decompose the semilocal ring $\Bbb Z_p[G]$ into a direct sum of local, $p$--admissible rings
$$\Bbb Z_p[G]=\bigoplus_\chi\mathcal O_\chi,$$
then it is very clear that we get a decomposition of $R_\infty$ into a direct sum of local rings
$$R_\infty\cong\bigoplus_{\chi}\mathcal O_\chi\llbracket T_1, T_2, \dots \rrbracket,$$
of closed maximal ideals $M_\chi:=\overline{\langle \frak m_\chi, T_1, T_2, \cdots\rangle}$, where $\frak m_\chi:=(p, I_P)$ is the maximal ideal of $\mathcal O_\chi$. (For notations, see the paragraph below \eqref{dec-groupring-eq}.) From this description, it is clear that
$$\overline{\langle T_1, T_2, \dots\rangle}\subseteq {\rm Jac}(R_\infty),$$
where ${\rm Jac}(R_\infty)$ is the Jacobson radical of $R_\infty.$
\end{remark}
\begin{lemma}
The $R_\infty$--module $\nabla(L_\infty/k)$ is finitely generated.
\end{lemma}
\begin{proof}
Consider the $R_\infty$--ideal $J=\overline{\langle T_1, T_2, \dots\rangle}$ and note that $J$ is contained in the Jacobson radical of $R_\infty$. Now, \eqref{nabla-codescent} implies that, if we let $L:=L_{0,0}$, we have $R_\infty$--module isomorphisms
$$\nabla(L_{n,m}/k)/J\nabla(L_{n,m}/k)\cong \nabla(L/k),\qquad\text{ for all }(n,m),$$
compatible with the transition maps. Since all the $R_\infty$--modules involved in the isomorphisms above are finitely presented, therefore compact and Hausdorff, when taking a projective limit of the above isomorphisms, we obtain an isomorphism of $R_\infty$--modules
\begin{equation}\label{nabla-infty-codescent}\nabla(L_\infty/k)/\varprojlim_{n,m}(J\nabla(L_{n,m}/k))\cong \nabla(L/k).\end{equation}

Now, we claim that we have an equality of $R_\infty$--modules
\begin{equation}\label{lim-commutes-with-J}\varprojlim_{n,m}(J\nabla(L_{n,m}/k))=\overline{J\nabla(L_\infty/k)}.\end{equation}
The right--hand side is obviously included in the left--hand side. Now, take $(x_{n,m})_{n,m}$ in the left--hand side and, for each $(n,m)$, let $C_{n,m}$ be the preimage of $x_{n,m}$ via the projection map
$$\pi_{n,m}: \overline{J\nabla(L_\infty/k)}\to J\nabla(L_{n,m}/k).$$
Then, $C_{n,m}$ is non--empty (as the transition maps are surjective), closed, and $C_{n',m'}\subseteq C_{n,m}$, whenever $n'\geqslant n$ and $m'\geqslant m$. Now, since $\overline{J\nabla(L_\infty/k)}$ is compact, Hausdorff (as a closed submodule of $\nabla(L_\infty/k)$), we have 
$$\bigcap_{n,m} C_{n,m}\ne\emptyset.$$
Any $x$ in the intersection above satisfies $x=(x_{n,m})_{n,m}$, by definition, which shows that the intersection is in fact a singleton and also concludes the proof of \eqref{lim-commutes-with-J}. 

When combining \eqref{lim-commutes-with-J} and \eqref{nabla-infty-codescent}, one obtains an isomorphism of $R_\infty$--modules 
$$\nabla(L_\infty/k)/\overline{J\nabla(L_\infty/k)}\cong\nabla(L/k).$$
Since $\nabla(L/k)$ is finitely generated as an $R_\infty/J$--module, Lemma \ref{top-Nakayama} implies that $\nabla(L_\infty/k)$ is finitely generated as an $R_\infty$--module. This concludes the proof of the Lemma.
\end{proof}    
Now, note that isomorphism \eqref{nabla-codescent} says that the surjective, projective system of $R_{n,m}$--modules $(\nabla(L_{n,m}/k))_{n,m}$ is well--structured. Now, the previous Lemma allows us to apply Corollary \ref{corollary-theorem2} to this projective system and obtain an equality of $R_\infty$--ideals
$$\overline{\fitzero^0_{R_\infty}\nabla(L_\infty/k)}=\varprojlim_{n,m}\,\fitzero^0_{R_{n,m}}\nabla(L_{n,m}/k).$$
Now, \eqref{nabla-fitt} combined with an argument similar to the one used to prove \eqref{lim-commutes-with-J} shows that
$$\varprojlim_{n,m}\,\fitzero^0_{R_{n,m}}\nabla(L_{n,m}/k)=\varprojlim_{n,m}\, (R_{n,m}\Theta_{L_{n,m}/k})=\langle\Theta_{L_\infty/k}\rangle.$$
Consequently, we obtain the following equality of $R_\infty$--ideals
\begin{equation}\label{almost-2}\overline{\fitzero^0_{R_\infty}\nabla(L_\infty/k)}=\langle\Theta_{L_\infty/k}\rangle.\end{equation}
Next, we need the following general topological algebra result.
\begin{lemma}\label{principal-closure}
Let $R$ be a compact, Hausdorff topological ring, whose maximal ideals are all closed. Let $I$ be an ideal in $R$ such that $\overline I=aR$, where $a$ is a non zero--divisor in $R$. Then, $I$ is closed and $I=aR$.
\end{lemma}
\begin{proof}
Let us consider the $R$--ideal $J:=\{x\mid x\in R, ax\in I\}$. Since $I\subseteq \overline I=aR$, one can find a net $(x_\alpha)_{\alpha\in\mathscr A}$ inside $J$, such that 
$$a=\lim\limits_\alpha\, (ax_\alpha).$$
Since $R$ is compact, $(x_{\alpha})_{\alpha\in\mathscr A}$ has a cluster point. So, without loss of generality, we may assume that $\lim\limits_{\alpha}\,x_\alpha$ exists. It is unique because $R$ is Hausdorff. (See Proposition \ref{uniquenetlimit} for all these facts.) Now, the last displayed equality combined with the fact that $a$ is not a zero--divisor implies that 
$$1=\lim\limits_{\alpha}\,x_\alpha\, \in \overline J.$$
This implies that $\overline J=R$. If $J\ne R$, then $J\subseteq \frak m$, for some maximal ideal $\frak m$ of $R$. However, $\frak m$ is closed, therefore $\overline J\subseteq \frak m$, which is a contradiction. Therefore, we have $J=R$ and 
$$Ra=Ja\subseteq I\subseteq \overline I=Ra,$$
which implies that $I=\overline I=Ra$, concluding the proof of the Lemma.
\end{proof}
Now, the above Lemma, combined with Proposition \ref{Bley-Pop-EMC} and equality \eqref{almost-2}, gives
\begin{equation}\label{part2}{\fitzero^0_{R_\infty}\nabla(L_\infty/k)}=\langle\Theta_{L_\infty/k}\rangle,
\end{equation}
which proves part (2) of Theorem \ref{main-th-appB}.

\begin{lemma}\label{finite-pres-lemma}
Assume that $R$ is a commutative, semilocal ring and $M$ is a finitely generated $R$--module whose Fitting ideal is principal, generated by a non zero--divisor. 
Then, $M$ is a finitely presented $R$--module.
\end{lemma}
\begin{proof}
Since $M$ is finitely generated, we have an exact sequence of $R$--modules
$$0\to K\to R^n\to M\to 0,$$
for some $n\geqslant 1.$ Assume that $\fitzero_R^0{M}=aR$, for a non zero--divisor $a$ in $R$. By definition, 
$$aR=\langle \Delta_n(K\subseteq R^n)\rangle,$$
where $\Delta_n(K\subseteq R^n)$ is the set of determinants $\det (\vec{k_1}, \dots, \vec{k_n})$ of $n\times n$--matrices, whose  columns $\vec{k_i}$ are in $K$. (View elements in $R^n$ as column vectors and $K$ as a subset of $R^n$.)

Now, take a maximal ideal $\frak m$ of $R$. Since Fitting ideals commute with localization, we have an equality of ideals in the local ring $R_{\frak m}$: 
$$aR_{\frak m}=\langle \Delta_n(K_{\frak m}\subseteq R_{\frak m}^n)\rangle .$$
Since $R_{\frak m}$ is local and $a$ is a non zero--divisor, this means that at least one of the generators of the right side $\det (\vec{k_1}, \dots, \vec{k_n})$, with $\vec{k_i}\in K_{\frak m}$, satisfies
$$\det (\vec{k_1}, \dots, \vec{k_n})=ua,\qquad \text{ with } u\in R_{\frak m}^\times.$$
Now, let $A=(\vec{k_1},\dots, \vec{k_n})$ and let $\vec x\in K_{\frak m}$ be arbitrary. By Cramer's rule we have 
$$ua\cdot\vec{x}=\det(A)\cdot\vec{x}=\sum_{i=1}^n\det(D_i)\cdot\vec{k_i},$$
where $D_i$ is obtained from $A$ by replacing its $i$--th column with $\vec{x}$. Therefore, $\det(D_i)\in\Delta_n(K_{\frak m}\subseteq R_{\frak m}^n)$ and, consequently, $\det(D_i)=a\delta_i$, for some $\delta_i\in R_{\frak m}$. Therefore, the last displayed equality can be rewritten
$$ua\cdot\vec{x}=a\cdot \sum_{i=1}^n \delta_i\cdot\vec{k_i}.$$
Now, since $\vec{x}$ is arbitrary, $a$ is a non zero--divisor and $u$ is a unit, this implies that 
$$K_{\frak m}=\langle \vec{k_1}, \dots, \vec{k_n}\rangle_{R_{\frak m}},$$
so $K_{\frak m}$ is finitely generated as an $R_{\frak m}$--module, for every maximal ideal $\frak m$ in $R$. Since $R$ is semilocal (meaning that $R$ has only finitely many maximal ideals $\frak m$), this implies right away that one can construct a finitely generated $R$--submodule $K_0$ of $K$, such that $(K/K_0)_{\frak m}=0$, for all maximal ideals $\frak m$. This shows that $K=K_0$ and therefore $K$ is finitely generated as an $R$--module. Consequently, $M$ is a finitely presented $R$--module.
\end{proof}
The next Lemma is due essentially to Cornacchia and Greither \cite{C-G}, who stated and proved it in the local case. The (easy) extension to the semilocal case was stated and proved as Proposition 4.9 in \cite{Gambheera-Popescu}.
\begin{lemma}[Cornacchia--Greither]\label{proj-dim} Let $R$ be a semilocal, commutative ring and let $M$ be a finitely presented $R$--module. Then, the following are equivalent.
\begin{enumerate}
\item There is a short exact sequence of $R$--modules
$$0\to R^n\to R^n\to M\to 0,$$
in other words, $M$ is quadratically presented and of projective dimension $1$.
\item The Fitting ideal of $M$ is principal, generated by a non zero--divisor.
\end{enumerate}
\end{lemma}
\medskip

Now, we are ready to prove part (1) of Theorem \ref{main-th-appB}. 

First, combine \eqref{part2} with Proposition \ref{Bley-Pop-EMC} and Lemma \ref{finite-pres-lemma} applied to the semilocal ring $R_\infty$ and the finitely generated $R_\infty$--module $\nabla(L_\infty/k)$, to conclude that $\nabla(L_\infty/k)$ is a finitely presented $R_\infty$--module.

Next, combine these facts with Lemma \ref{proj-dim}, to conclude that $\nabla(L_\infty/k)$ is quadratically presented and of projective dimension $1$ as an $R_\infty$--module. 

Further, the $R_\infty$--module $\nabla(L_\infty/k)$ is torsion because its Fitting ideal is not trivial and, consequently, its annihilator is not trivial. 
\end{proof}

\end{appendices}

\end{document}